\documentclass{amsart}
\usepackage{latexsym}
\usepackage{amsmath}
\usepackage{amssymb}
\usepackage[all]{xy}
\usepackage{autobreak}
\usepackage{enumitem}
\usepackage[hypertexnames=false]{hyperref} 

\usepackage[\ifdefined\enabletodonotes\else disable\fi]{todonotes}

\usepackage{autonum} 

\usepackage{tikz}
\usetikzlibrary{positioning}
\usetikzlibrary{arrows.meta}
\usetikzlibrary{calc}

\numberwithin{equation}{section}
\newtheorem{thm}[equation]{Theorem}
\newtheorem{prop}[equation]{Proposition}

\newtheorem{cor}[equation]{Corollary}
\newtheorem{lem}[equation]{Lemma}
\theoremstyle{definition}
\newtheorem{defn}[equation]{Definition}

\newtheorem{assertion}[equation]{Assertion}

\theoremstyle{remark}
\newtheorem{rem}[equation]{Remark}
\newtheorem{ex}[equation]{Example}

\newcommand{\K}{{\mathbb K}}
\newcommand{\Z}{{\mathbb Z}}
\newcommand{\C}{{\mathbb C}}
\newcommand{\Q}{{\mathbb Q}}
\newcommand{\calH}{{\mathcal H}}
\newcommand{\diag}{\operatorname{Diag}}
\renewcommand{\ker}{\operatorname{Ker}}
\newcommand{\im}{\operatorname{Im}}
\newcommand{\coker}{\operatorname{Coker}}
\newcommand{\rank}{\operatorname{rank}}
\newcommand{\comp}{\operatorname{Comp}}
\newcommand{\ext}{\operatorname{Ext}}
\newcommand{\apl}{A_{PL}}
\newcommand{\dlp}{\operatorname{Dlp}}
\newcommand{\dlcop}{\operatorname{Dlcop}}
\newcommand{\dsb}{\operatorname{Dsb}}
\newcommand{\HochschildModel}{\mathcal L}
\newcommand{\HochschildSub}{\mathcal L'}
\newcommand{\cyclicModel}{\mathcal E}
\newcommand{\pathspaceModel}{\mathcal P}
\newcommand{\redHochschildModel}{\tilde\HochschildModel}
\newcommand{\redHochschildComponent}[1]{\redHochschildModel^{(#1)}}
\newcommand{\redCyclicModel}{\tilde\cyclicModel}
\newcommand{\susp}{s}
\newcommand{\redSusp}{\tilde{\susp}}
\newcommand{\spl}{\operatorname{sp}}
\renewcommand{\deg}[1]{\mathopen\vert #1\mathclose\vert}
\newcommand{\sdeg}[1]{\mathopen\parallel #1\mathclose\parallel}
\newcommand{\mapright}[1]{%
 \smash{\mathop{%
  \hbox to 1cm{\rightarrowfill}}\limits_{#1}}}
\newcommand{\maprightd}[2]{%
 \smash{\mathop{%
  \hbox to 1.2cm{\rightarrowfill}}\limits^{#1}\limits_{#2}}}
\newcommand{\mapleft}[1]{%
 \smash{\mathop{%
  \hbox to 1cm{\leftarrowfill}}\limits_{#1}}}
\newcommand{\mapleftu}[1]{%
 \smash{\mathop{%
  \hbox to 0.8cm{\leftarrowfill}}\limits^{#1}}}
\newcommand{\maprightu}[1]{%
 \smash{\mathop{%
  \hbox to 1cm{\rightarrowfill}}\limits^{#1}}}
\newcommand{\maprightud}[2]{%
 \smash{\mathop{%
  \hbox to 1cm{\rightarrowfill}}\limits^{#1}_{#2}}}
\newcommand{\mapleftud}[2]{%
 \smash{\mathop{%
  \hbox to 1cm{\leftarrowfill}}\limits^{#1}_{#2}}}

\title[A reduction of the string bracket to the loop product]{A reduction of the string bracket to the loop product}

\author[K. Kuribayashi]{Katsuhiko Kuribayashi}
\address{%
  Department of Mathematical Sciences,
  Faculty of Science,
  Shinshu University,
  Matsumoto, Nagano 390-8621, Japan
}
\email{kuri@math.shinshu-u.ac.jp}

\author[T. Naito]{Takahito Naito}
\address{%
  Nippon Institute of Technology,
  Gakuendai, Miyashiro-machi, Minamisaitama-gun, Saitama 345-8501, Japan
}
\email{naito.takahito@nit.ac.jp}

\author[S. Wakatsuki]{Shun Wakatsuki}
\address{%
  Department of Mathematical Sciences,
  Faculty of Science,
  Shinshu University,
  Matsumoto, Nagano 390-8621, Japan
}
\email{swaka@shinshu-u.ac.jp}

\author[T. Yamaguchi]{Toshihiro Yamaguchi}
\address{%
  Faculty of Education,
  Kochi University, Akebono-cho, Kochi 780-8520, Japan
}
\email{tyamag@kochi-u.ac.jp}

\thanks{%
  The first author was partially supported by a Grant-in-Aid for Scientific
  Research (B) 21H00982 from Japan Society for the Promotion of Science.
  The second author was supported by JSPS KAKENHI Grant Number JP18K13403.
  The third author was supported by JSPS KAKENHI Grant Number 20J00404.
  The fourth author was partially supported by JSPS KAKENHI Grant Number 20K03591.
}

\begin{document}

\maketitle

\begin{abstract}
The negative cyclic homology for a differential graded algebra over the rational field has a quotient of the Hochschild homology as a direct summand if the $S$-action is trivial.
With this fact, we show that the string bracket in the sense of Chas and Sullivan is reduced to the loop product followed by
the BV operator on the loop homology provided the given manifold is {\it BV-exact}. The reduction is indeed derived from the equivalence between the BV-exactness and the triviality of the $S$-action. Moreover, it is proved that a Lie bracket on the loop cohomology of the classifying space of a connected compact Lie group possesses the same reduction.
By using these results, we consider the non-triviality of string brackets. We also show that a simply-connected space with positive weights is BV-exact. Furthermore, the {\it higher BV-exactness} is discussed featuring the cobar-type Eilenberg-Moore spectral sequence.
\end{abstract}


{
  \makeatletter
  \providecommand\@dotsep{5}
  \makeatother
}

\setcounter{tocdepth}{1} 
\tableofcontents

\section{Introduction}
Let $LM$ be the free loop space, namely, the space of continuous maps from the circle $S^1$ to a space $M$ with compact-open topology. The rotation on the domain space $S^1$ of $LM$ induces an  $S^1$-action on $LM$. Then we have  the $S^1$-equivariant homology $H^{S^1}_*(LM)=H_*(ES^1\times_{S^1}LM)$ for a space $M$.
The {\it string bracket} is a Lie bracket on the $S^1$-equivariant homology
of the free loop space $LM$ of an orientable closed manifold $M$,
which is introduced by Chas and Sullivan in \cite{C-S}.
The bracket is defined by using the loop product
on the loop homology $H_*(LM)$ and maps in the Gysin exact sequence of the $S^1$-principal bundle
\begin{equation}
  S^1 \to ES^1\times LM \to ES^1\times_{S^1} LM.
\end{equation}
In particular, the Batalin--Vilkovisky (BV) identity of the BV operator on the loop homology induces the Jacobi identity for the string bracket; see the proof of
\cite[Theorem 6.1]{C-S}.

As for computations of the string brackets, Basu \cite{B} and F\'elix, Thomas and Vigu\'e-Poirrier \cite{F-T-V07} have determined explicitly the rational string bracket of the product of spheres. For a simply-connected closed manifold $M$ whose rational cohomology is generated by a single element,
the rational string bracket is trivial though the rational loop product of $M$ is highly non-trivial; see \cite[Theorem 3.4]{B} and \cite[5.2. Example 1)]{F-T-V07}. On the other hand, a result due to Tabinmg \cite{Tab} shows that the integral string bracket of the sphere is non-trivial.

The loop homology of the classifying space $BG$ of a connected compact Lie group $G$ in the sense of Chataur and Menichi \cite{C-M}
admits the BV algebra structure, see also \cite[Theorem C.1]{K-M}.
Therefore, the same argument as that about manifolds allows us to deduce that the string cohomology of $BG$
is endowed with a graded Lie algebra structure; see Proposition \ref{prop:LBG} and \cite[Theorem 1.1]{CEL19}.

The aim of this article is to investigate general methods for computing the rational string brackets
for a manifold and the classifying space of a connected compact Lie group. The key strategy is to use Jones' isomorphisms
\[
H^*(LM; \Q) \cong HH_*(\apl(M)) \ \  \text{and} \ \  H^*_{S^1}(LM; \Q) \cong HC_*^-(\apl(M)),
\]
where $\apl(M)$ is the polynomial de Rham algebra over $\Q$ of a simply-connected space $M$ and the right-hand sides of the isomorphisms denote the Hochschild homology and the negative cyclic homology of the complex, respectively; see Section \ref{sect:HH-HC} for more details.
Furthermore, the decomposition theorem of the negative cyclic homology and the cyclic homology (additive K-theory in the sense of Feigin and Tsygan \cite{F-T}) in \cite{VP, K-Y} is applied in the computation; see Theorem \ref{thm:brackets}.
It turns out that for a simply-connected closed manifold $M$,
the rational string bracket for $M$ is reduced to the loop product of $M$ followed by the BV operator provided the manifold possesses the exactness of the operator; see Definition \ref{defn:BV-exact}.%

\begin{assertion}\label{assertion:HC_HH} Let $M$ be a simply-connected closed manifold. Suppose further that $M$ is {\it BV-exact}. Then
the string bracket in the string homology $H_*^{S^1}(LM; \Q)$
is regarded as the {\it loop bracket} in the loop homology $H_*(LM; \Q)$ up to isomorphism and hence the string bracket is determined by the Gerstenhaber bracket in the Hochschild cohomology
of  the polynomial de Rham algebra  $\apl(M)$ of $M$.
\end{assertion}

The detail is described in Corollary \ref{cor:stringbracket_loopbracket}. In particular, the nilpotency of the string bracket is equivalent to that of
the Gerstenhaber bracket. We stress that
the Gerstenhaber algebra in Assertion \ref{assertion:HC_HH} is considered with the Lie model for $M$ without using the loop product; see \cite{F-M-T05}.
It is worth mentioning that the BV-exactness, which is introduced to consider the reduction of the string brackets, is a new homotopy invariant deeply related to other traditional rational homotopy invariants for spaces. We discuss and summarize this topic in Assertion \ref{assertion:implications} below and several paragraphs before the assertion.

F\'elix, Thomas and Vigu\'e-Poirrier \cite{F-T-V07} have given an explicit description of the rational string bracket of $M$ with its
Sullivan model. On the other hand, our method for computing the string bracket is formulated with the loop product and the BV operator
on the loop homology. Moreover, the BV-exactness is also described in terms of the loop homology.
Therefore, it is possible to make a computation of the dual to the string bracket
on the equivariant homology \(H_*^{S^1}(LM; \Q)\)
by considering {\it only} behavior of
the BV operator on the loop homology \(H_*(LM; \Q)\); see Remark \ref{rem:Strategy} for more details. This is an advantage of our result.

In case of the classifying space, the same strategy as above is applicable in the computation of the string baracket.
In fact, for the classifying space $BG$ of
{\it every} compact connected Lie group $G$, the rational string bracket for $BG$ is described as the BV operator followed by the dual loop coproduct; see Theorems \ref{thm:main} (i) and \ref{thm:manifolds} (i). As for general properties of the string brackets, the theorems allow us to deduce that the Lie bracket on
the string cohomology $H^*_{S^1}(LBG; {\mathbb Q})$ is highly non-trivial even if $\rank G = 1$; see Proposition \ref{prop:IteratedBrackets}. Moreover,
Propositions \ref{prop:nil} and \ref{prop:non-nil} assert that the loop homology endowed with the string bracket of a simply-connected Lie group $G$ is nilpotent if and only if $\rank G = 1$. 

The notion of a {\it Gorenstein space} due to F\'elix, Halperin and Thomas \cite{FHT_G} enables us to deal with a manifold and the classifying space of a Lie group simultaneously.
As a consequence, with the influence of string topology on Gorenstein spaces \cite{F-T_ST}, we have Theorems \ref{thm:main}, \ref{thm:manifolds} and \ref{thm:brackets} mentioned above.

We moreover propose a method for computing the string bracket of a non BV-exact space $M$.
To this end, we introduce a bracket on the cobar-type Eilenberg-Moore spectral sequence (EMSS) converging to $H^*_{S^1}(LM; \Q)$ which is compatible with the string bracket of the target; see Theorem \ref{thm:cobrackets}. Moreover,
the EMSS carries a decomposition compatible with the Hodge decomposition of the target; see Remark \ref{rem:HodgeDecom}.
While there is no computational example obtained by applying the spectral sequence, in future work, it is expected that the EMSS is applicable in computing the string bracket explicitly; see \ref{subsect:OpenQ} Problems.

As described above, the BV-exactness is a key to computing string brackets on Gorenstein spaces.
Moreover, it is worthwhile mentioning that the BV-exactness for a space $M$ is equivalent to the triviality of the $S$-action in Connes' exact sequence
; see Theorem \ref{thm:BV-S}.
In fact, the new invariant is only described in terms of the Hochschild homology while the $S$-action is defined on the negative cyclic homology.
A deep consideration due to Vigu\'e-Poirrier
in \cite{VP88,VP} shows that the $S$-action on the negative cyclic homology is trivial if $M$ is formal.
Thus we see that the class of BV-exact spaces contains that of formal spaces; see Corollary \ref{cor:BV}.

With historical perspectives, we comment on relationships among notions of
$p$-{\it universality} in \cite{MNT}, {\it positive weights} in \cite{BD}, the BV-exactness and its variants; see Definition \ref{defn:p.w.} for positive weights.

By definition, simply-connected spaces  $X$ and $Y$ are said to be $p$-{\it equivalent} if  there is a map $f: X\to Y$ which induces $H^*(X;\Z/p)\cong H^*(Y;\Z/p)$, where $p$ is a prime or zero and $\Z/0 = \Q$.
In \cite{Se}, Serre raised the so-called symmetry question whether the existence of a $p$-equivalence $X\to Y$ implies the existence of a $p$-equivalence in the reverse direction $Y\to X$.
However, in general,  the $p$-equivalence does not satisfy the symmetricity.

Mimura, O’Neil and Toda in \cite{MNT} defined the notion of a $p$-{\it universal} space and proved that
in the full subcategory of $p$-universal spaces of the category of simply-connected spaces whose homotopy types are those of finite CW complexes, the $p$-equivalence is indeed
an equivalence relation. We observe that the $p$-universality does not depend on $p$ or $0$; see \cite[Proposition 2.9]{MNT}.
Afterward, Body and Douglas \cite{BD} defined the concept of positive weights for Sullivan minimal models.
The result \cite[Theorem 2]{Sc} due to Scheerer, in turn, yields that the two notions of $p$-universality and positive weights are equivalent.

By using the EMSS mentioned above, we also introduce the notion of \textit{$r$-BV-exactness} (see Definition \ref{defn:r-BV}). The $r$-BV-exactness for a simply-connected space $M$ is equivalent to the collapsing at the $E_{r+1}$-term of the EMSS for $M$; see Corollary \ref{cor:collapsing}.
 The decomposition of the EMSS allows us to deduce that the notion of BV-exactness is indeed equivalent to that of $1$-BV-exactness; see Theorem \ref{thm:BV-S_general}.
Thus $r$-BV-exactness is regarded as a higher version of BV-exactness.
We summarize important relationships among invariants mentioned above.

\begin{assertion}\label{assertion:implications}
There are the following implications concerning rational homotopy invariants for a simply-connected space $X$.

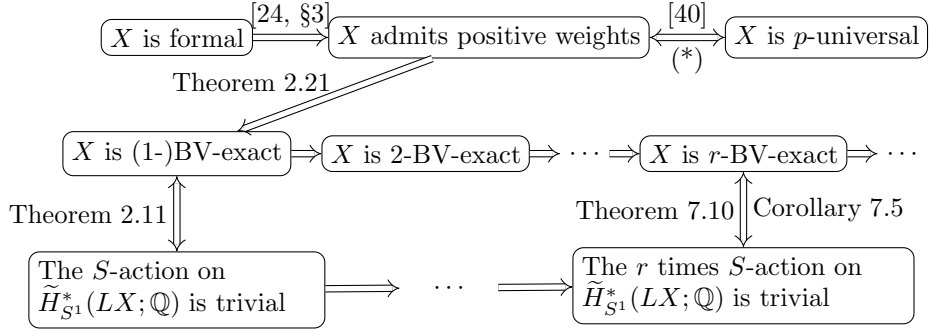
\hspace{-0.5cm}
\begin{tikzpicture}
  \tikzset{block/.style={draw,rounded corners}}
  \tikzset{implies/.style={-Implies, double, double distance=2pt}}
  \tikzset{equivalent/.style={Implies-Implies, double, double distance=2pt}}
  \newcommand{\len}{0.4cm}
  \node[block] (formal) at (0, 0) {\(X\) is formal};
  \node[block, right=of formal] (weight) {\(X\) admits positive weights};
  \node[block, right=of weight] (universal) {\(X\) is \(p\)-universal};
  \node[block, below=of formal] (1BV) {\(X\) is (1-)BV-exact};
  \node[block, right=\len of 1BV] (2BV) {\(X\) is 2-BV-exact};
  \node[right=\len of 2BV] (dots1) {\(\cdots\)};
  \node[block, right=\len of dots1] (rBV) {\(X\) is \(r\)-BV-exact};
  \node[right=\len of rBV] (dots2) {\(\cdots\)};
  \node[block, below=of 1BV, text width = 3.7cm] (S)
    {The \(S\)-action on \(\widetilde{H}^*_{S^1}(LX; \Q)\) is trivial};
  \node[block, below=of rBV, text width = 4.3cm] (Sr) {The \(r\) times \(S\)-action on \(\widetilde{H}^*_{S^1}(LX; \Q)\) is trivial};
  \node[minimum width=1.3cm] (dots3) at ($(S.east)!0.5!(Sr.west)$) {\(\cdots\)};
  \draw[implies] (formal) -- node[above]{\cite[\S3]{H-S}} (weight);
  \draw[implies] (weight) -- node[above left=-1mm]{Theorem \ref{thm:weightBVexact}} (1BV);
  \draw[implies] (1BV)--(2BV);
  \draw[implies] (2BV)--(dots1);
  \draw[implies] (dots1)--(rBV);
  \draw[implies] (rBV)--(dots2);
  \draw[equivalent] (1BV) -- node[left]{Theorem \ref{thm:BV-S}} (S);
  \draw[equivalent] (rBV) -- node[left]{Theorem \ref{thm:BV-S_general}} node[right]{Corollary \ref{cor:collapsing}} (Sr);
  \draw[equivalent] (weight) -- node[above]{\cite{Sc}} node[below]{(*)} (universal);
  \draw[implies] (S) -- (dots3);
  \draw[implies] (dots3) -- (Sr);
\end{tikzpicture}

\smallskip
\noindent
Here the reduced cohomology $\widetilde{H}^*_{S^1}(LX; \Q)$ is the cokernel of the map
$H_{S^1}^*(*; \Q) \to {H}_{S^1}^*(LX; \Q)$ induced by the trivial map
and the \(S\)-action on \(\widetilde{H}^*_{S^1}(LX; \Q)\) is defined by
the multiplication of the generator of $\widetilde{H}^*(BS^1;\Q)$ with the map induced by the projection $q$ of the fibration $LX \to ES^1\times_{S^1}LX \stackrel{q}{\to} BS^1$. Observe that the equivalence (*) holds if \(X\) has the homotopy type of a finite CW complex.
\end{assertion}

As mentioned above, a simply-connected space admitting positive weights
is BV-exact.  Proposition \ref{prop:11-mfd} gives an example of a nonformal BV-exact manifold. Moreover, we obtain an elliptic and non BV-exact space in Appendix \ref{app:appA}.

This manuscript is organized as follows. In Section \ref{sect:assertions}, our results are stated in detail. In Section \ref{sect:HH-HC}, we recall the Hochschild
homology, the cyclic homology and Connes' exact sequences. Moreover, the Gorenstein space in the sense of F\'elix, Halperin and Thomas \cite{FHT_G} is also recalled. Section \ref{sect:proofs} provides the proofs of our results described in Section \ref{sect:assertions}.
Section \ref{sect:computations} discusses the nilpotency of the string homology of a Lie group and the classifying space of a Lie group.
In Section \ref{sect:nonformal-example}, the BV-exactness for a non-formal manifold of dimension 11 is considered.
Thanks to the reduction for computing the bracket described in Section
\ref{sect:assertions}, we determine explicitly the dual string bracket for the manifold; see Theorem \ref{thm:An_explicit_computation}.
We believe that the result gives the first example  which computes the string bracket of a non formal space.
Section \ref{sect:EMSS_rBVexact} considers the cobar-type Eilenberg-Moore spectral sequence (EMSS) for computing string brackets of non BV-exact manifolds.

In Appendix \ref{app:appA}, we obtain an example of an elliptic and non BV-exact space.
Appendix \ref{app:appB} gives a description of the Gysin exact sequence associated with the principal bundle $S^1 \to ES^1\times LM \stackrel{p}{\to} ES^1\times_{S^1} LM$ for a simply-connected space $M$ in terms of Sullivan models; see \cite[(5.12) Theorem]{Whitehead} for the exact sequence.

\subsection{Problems}\label{subsect:OpenQ}
We propose questions and problems on topics in this article.
\begin{enumerate}[label={P\arabic{enumi}.}]
  \item If a space is BV-exact, then does it admit positive weights?
  \item For each $r>1$, is there an $r$-BV-exact space which is not $(r-1)$-BV-exact?
  \item Is a space $r$-BV-exact for some $r<\infty$?
  \item
  By making use of the EMSS in Section \ref{sect:EMSS_rBVexact},
  compute explicitly the string brackets of a non BV-exact manifold. 
\end{enumerate}

\subsection{List of notations}
We list some notations used repeatedly in this article.

\begin{figure}[h]
  \begin{center}
    \begin{tabular}{|c|p{8 cm} c|}
      \hline
      \(\bullet \) & the loop product & \pageref{index:loopPro} \\
       \(\odot \) & the dual loop coproduct & \pageref{index:Dcoproduct} \\
        \([\ ,\ ]\) & the string bracket, dual string cobracket & \pageref{index:stringBracket} \pageref{index:StringBracket2}\\
      \(\Delta \) & the BV operator on the Hochschild homology of a differential graded algebra & \pageref{index:CoHomologicalBV} \\
       \(\Delta' \) & the BV operator on the homology of \(LM\) & \pageref{index:HomologicalBV} \\
      \(HH_*(\Omega)\) & the Hochschild homology of a DGA \(\Omega\) & \pageref{index:HochschildHomology}\\
      \(\widetilde{HH}_*(\Omega)\) & the reduced Hochschild homology, \(HH_*(\Omega) \cong \widetilde{HH}_*(\Omega)\oplus \K\) & \pageref{index:reducedHochschildHomology} \\
      \(HC^-_*(\Omega)\) & the negative cyclic homology of a DGA \(\Omega\) & \pageref{index:negativeCyclicHomology}\\
      \(\widetilde{HC}_*^-(\Omega)\) & the reduced negative cyclic homology, \(HC_*^-(\Omega) \cong \widetilde{HC}_*^-(\Omega)\oplus\K[u]\) & \pageref{index:reducedNegativeCyclicHomology} \\
      \(S\) & the \(S\)-action on the negative cyclic homology & \pageref{index:S-action} \\
      \(\HochschildModel\),  \((\HochschildModel, \delta )\) & the Sullivan minimal model for the free loop space \(LM\) (and the Hochschild homology) & \pageref{index:HochschildModel} \\
      \(\cyclicModel\), \((\cyclicModel, D)\) & the Sullivan minimal model for the Borel construction \(ES^1\times_{S^1} LM\) (and the negative cyclic homology)& \pageref{index:cyclicModel}\\
      \((\redHochschildModel, \delta)\) & the reduced version of \((\HochschildModel, \delta)\) & \pageref{index:redHochschildModel}\\
      \((\redHochschildComponent{n}, \delta)\) & a direct summand of \((\redHochschildModel, \delta) = \bigoplus_n(\redHochschildComponent{n}, \delta)\) & \pageref{index:redHochschildComponent} \\
      \(s\) & a derivation on \(\HochschildModel\), which is a chain model of \(\Delta\) & \pageref{index:s} \\
      \hline
    \end{tabular}
  \end{center}
\end{figure}

\section{String brackets described in terms of the Hochschild homology }\label{sect:assertions}

While the underlying field in Proposition  \ref{prop:LBG} below is of {\it arbitrary} characteristic, other results described in this section hold for a field of characteristic zero.

Let \(\K\) be a field and
denote the singular homology and cohomology with coefficients in \(\K\)
by \(H_*(-)\) and \(H^*(-)\), respectively.
For an orientable closed manifold $M$ of dimension $d$,
the Chas and Sullivan loop product $\bullet$ on the shifted homology ${\mathbb H}_*(LM) := H_{*+d}(LM)$ is unital, associative and graded commutative; see \cite{C-S}.
Consider the principal bundle $S^1 \to ES^1\times LM \stackrel{p}{\to} ES^1\times_{S^1} LM$.
The bundle gives rise to the homology Gysin sequence
\[
\xymatrix{
\cdots
\to
{\mathbb H}_{*-d}(LM)
\ar[r]^-{p_*}
&
H_*^{S^1}(LM)
\ar[r]^-{c}
&
H_{*-2}^{S^1}(LM)
\ar[r]^-{\text{M}}
&
{\mathbb H}_{*-d-1}(LM)
\to
\cdots.
}
\]
The \textit{string bracket} $[ \ , \ ]$\label{index:stringBracket}
on $H_*^{S^1}(LM)$ is defined by
\begin{equation}\label{eq:StringBracket}
[a, b] := (-1)^{\deg{a}-d}p_*(\text{M}(a)\bullet \text{M}(b))
\end{equation}
for $a, b \in H_*^{S^1}(LM)$. Observe that the bracket is of degree $2-d$ and gives a Lie algebra structure to the equivariant homology of $LM$.

Let $G$ be a connected compact Lie group of dimension $d$.
We write ${\mathbb H}^*(LBG) := H^{*+d} (LBG)$ and ${\mathcal H}^*(LBG) :=H^{*+d+1}_{S^1}(LBG)$.
With the notation, the cohomology Gysin sequence associated with the principal bundle $S^1 \to ES^1\times LBG \stackrel{p}{\to} ES^1\times_{S^1} LBG$
induces an exact sequence of the form
\[
\xymatrix{
\cdots \to
{\mathcal H}^{*-2}(LBG)
\ar[r]^-{S}
&
{\mathcal H}^{*}(LBG)
\ar[r]^-{\pi:=p^*}
&
{\mathbb H}^{*+1}(LBG)
\ar[r]^-{\beta}
&
{\mathcal H}^{*-1}(LBG)
\to
\cdots.
}
\]
Chataur and Menichi \cite{C-M} have proved that there exists an associative and graded commutative multiplication $\odot$ on ${\mathbb H}^* (LBG)$ which is induced by the dual loop coproduct with an appropriate sign; see \cite[Corollary B.3]{K-M} and also Section \ref{sect:HH-HC}.
Then the \textit{dual string cobracket} $[ \ , \ ]$ \label{index:StringBracket2}
on ${\mathcal H}^*(LBG)$ is defined by
\begin{equation}\label{eq:StringBracket2}
[x, y] := (-1)^{\sdeg{x}}\beta (\pi(x)\odot \pi(y))
\end{equation}
for $x, y \in {\mathcal H}^*(LBG)$.
Here the notation $\sdeg{x}$ means the degree of $x$ as an element in the shifted cohomology.

\begin{prop}\label{prop:LBG} 
Let $G$ be a connected compact Lie group of dimension $d$ and $\K$ a field of arbitrary characteristic.
Then the dual string cobracket gives ${\mathcal H}^*(LBG)$ a graded Lie algebra structure.
\end{prop}

\begin{rem}
Proposition \ref{prop:LBG} is a particular case of \cite[Theorem 65]{C-M} and \cite[Theorem 1.1]{CEL19}.
The result \cite[Theorem 65]{C-M} shows the Lie algebra structure on a homological conformal field theory.
The result \cite[Theorem 1.1]{CEL19} describes a gravity algebra structure on the negative cyclic homology of a mixed complex; see \cite{G} for a gravity algebra.
We give an elementary proof of this proposition by taking care of sign convention in Section \ref{sect:proofs}.
\end{rem}

We relate the string brackets (i.e., the string bracket \eqref{eq:StringBracket} and the dual string cobracket \eqref{eq:StringBracket2}) above to the Hochschild homology and the cyclic homology.
Let $\Omega$ be a connected differential graded algebra (DGA) over a field $\K$ of arbitrary characteristic. A DGA $\Omega$ is called a {\it cochain algebra} if the differential is of degree $+1$. If the differential of a DGA $\Omega$ decreases degree by one, we call the DGA $\Omega$ a {\it chain algebra}.
Let $\Omega$ be a chain algebra, which is nonpositive; that is, $\Omega = \oplus_{i\leq 0}\Omega_i$.
We recall Connes' exact sequences \cite[Theorem 2.2.1 and Proposition 5.1.5]{Loday} for the Hochschild homology, cyclic homology and the negative cyclic homology of $\Omega$, which are of the form
\begin{equation}\label{eq:Connes'ExactSe}
\xymatrix@C18pt@R20pt{
\cdots \ar[r] &HH_{n+1}(\Omega) \ar[r]^{I} & HC_{n+1}(\Omega) \ar[r]^{S'} & HC_{n-1}(\Omega) \ar[r]^{B_{HH}} & HH_n(\Omega) \ar[r]  &\cdots ,
}
\end{equation}
\[
\xymatrix@C18pt@R15pt{
\cdots \ar[r] & HC_{n+2}^-(\Omega) \ar[r]^{S=\times u} & HC_{n}^-(\Omega) \ar[r]^{\pi} & HH_{n}(\Omega) \ar[r]^{\beta} & HC_{n+1}^-(\Omega) \ar[r] &\cdots \ \
\text{and}
}
\]
\[
\xymatrix@C18pt@R15pt{
\cdots \ar[r] & HC_{n+1}^-(\Omega) \ar[r]^{\times u} & HC_{n-1}^{\text{per}}(\Omega) \ar[r]^{\tilde{\pi}} & HC_{n-1}(\Omega) \ar[r]^{B_{HC}} & HC_{n}^-(\Omega) \ar[r] &\cdots .
}
\]
Here $S$ denotes the $S$-action and the maps $B_{HH}$, $\beta$ and $B_{HC}$ are induced by Connes' $B$-map $B$; see Section \ref{sect:HH-HC-1} for more details.
The reduced versions of the Hochschild homology and the negative cyclic homology of $\Omega$ are denoted by $\widetilde{HH}_*(\Omega)$ and $\widetilde{HC}_*^-(\Omega)$, respectively (see Section \ref{sect:HH-HC-1}).


\begin{rem}\label{rem:PositiveToNegative}
Following Jones \cite{J}, we define the Hochschild homology and the cyclic homology for a chain algebra but not a cochain algebra. For a cochain algebra $\Omega$, we define a chain algebra $\Omega_\sharp$
by $(\Omega_\sharp)_{-i} = \Omega^i$ for $i$.
Thus, for a nonnegative cochain algebra ${\mathcal M}$,
we have a nonpositive chain algebra ${\mathcal M}_\sharp$.
The Hochschild homology and the negative cyclic homology of ${\mathcal M}$
are defined by $HH_*({\mathcal M}_\sharp)$ and $HC_*^-({\mathcal M}_\sharp)$, respectively.
By abuse of notation, we may write $HH_*({\mathcal M})$ and $HC_*^-({\mathcal M})$ for $HH({\mathcal M}_\sharp)$ and $HC_*^-({\mathcal M}_\sharp)$, respectively.
\end{rem}

The constructions of the string brackets above are generalized with \textit{Gorenstein spaces}.
An orientable manifold and the classifying space of a connected Lie group are
typical examples of Gorenstein spaces; see Section \ref{sect:HH-HC} for the definition and fundamental properties of a Gorenstein space.
For a Gorenstein space $M$ of dimension $d$, we define a comultiplication $\bullet^\vee$
and a multiplication $\odot$ on the cohomology $H^*(LM; \K)$ which are called the {\it dual loop product} and the {\it dual loop coproduct}, respectively;
see Section \ref{sect:HH-HC}.  Therefore, by using the formulae \eqref{eq:StringBracket} and \eqref{eq:StringBracket2} above, we have the string bracket and the dual string cobracket for a Gorenstein space $M$ with $\bullet := (\bullet^\vee)^\vee$ and $\odot$, respectively; see Theorems \ref{thm:main} and \ref{thm:manifolds} below
for more details.
We do not know the string brackets satisfy the Jacobi identity for general Gorenstein spaces.
However, as seen in Theorem \ref{thm:manifolds}, these constructions indeed give generalizations of brackets \eqref{eq:StringBracket} on manifolds and \eqref{eq:StringBracket2} on classifying spaces.

The following theorem asserts that the dual to the string bracket in the sense of Chas and Sullivan for a manifold is the dual loop product followed by the BV operator.
Moreover, we see that the string bracket in Proposition \ref{prop:LBG} is described as the BV operator followed by the dual loop coproduct.

In the rest of this section,
we further assume that \(\K\) is a field of characteristic zero and
a DGA $\Omega$ is locally finite; that is the homology $H_i(\Omega)$ is finite dimensional
for each $i \leq 0$.

\begin{thm}\label{thm:main}
Let $M$ be a simply-connected Gorenstein space and $\Omega$ the chain algebra $\apl(M)_\sharp \otimes_{\Q}\K$.
Suppose that the $S$-action on the reduced negative cyclic homology $\widetilde{HC}^-_{*}(\Omega)$ is trivial. Then
one has the following assertions {\em (i)} and {\em (ii)}.\\
{\em (i)} There is a commutative diagram
\[
\xymatrix@C30pt@R12pt{
  ((\widetilde{HH}_{*}(\Omega)/\im \Delta) \oplus \K[u])^{\otimes 2} \ar[r]_-{\cong}^-{\Xi\otimes \Xi} \ar[d]_{\Delta\otimes\Delta}
  & HC_*^-(\Omega)^{\otimes 2} \ar[d]_{\pi\otimes\pi} \\
  HH_*(\Omega)^{\otimes 2} \ar[d]_-{\odot} & HH_*(\Omega)^{\otimes 2} \ar[d]_-{\odot} \\
  HH_*(\Omega) \ar[d]_{\text{\em `Cokernel'}} & HH_*(\Omega) \ar[d]_{\beta}\\
  (\widetilde{HH}_{*}(\Omega)/\im \Delta) \oplus \K[u] \ar[r]_-{\cong}^-{\Xi} & HC_*^-(\Omega).
}
\]
Here \(\Delta = B_{HH}\circ I\colon HH_*(\Omega) \to HH_*(\Omega)\) \label{index:CoHomologicalBV} is the ``BV operator'', 
$\odot$ is the product described in Section \ref{subsect:loopprod-coprod}, 
$\text{\em `Cokernel'}$ is defined by $(\text{\em The projection on the cokernel}, 0)$
and the horizontal isomorphism $\Xi$ is defined by the composite
\[
  (\widetilde{HH}_{*}(\Omega)/\im\Delta) \oplus \K[u] \xrightarrow[\cong]{I}
  \widetilde{HC}_{*}(\Omega) \oplus \K[u] \xrightarrow[\cong]{B_{HC}}
  \widetilde{HC}_{*}^-(\Omega) \oplus \K[u] \xrightarrow[\cong]{\spl}
  HC_*^-(\Omega)
\]
with the map $\spl$ in Remark \ref{rem:split} below.

\noindent
{\em (ii)} There is a commutative diagram
\[
\xymatrix@C30pt@R12pt{
  ((\widetilde{HH}_{*}(\Omega)/\im \Delta) \oplus \K[u])^{\otimes 2} \ar[r]_-{\cong}^-{\Xi\otimes \Xi} & HC_*^-(\Omega)^{\otimes 2}  \\
  HH_*(\Omega)^{\otimes 2} \ar[u]^{\text{\em `Cokernel'}\otimes\text{\em `Cokernel'}} & HH_*(\Omega)^{\otimes 2} \ar[u]^{\beta\otimes\beta} \\
  HH_*(\Omega) \ar[u]^-{\bullet^\vee} & HH_*(\Omega) \ar[u]^-{\bullet^\vee}\\
  (\widetilde{HH}_{*}(\Omega)/\im \Delta) \oplus \K[u] \ar[r]_-{\cong}^-{\Xi} \ar[u]^{\Delta} & HC_*^-(\Omega). \ar[u]^{\pi}
}
\]
 Here \(\Delta = B_{HH}\circ I\) is the BV operator of the BV algebra \(HH_*(\Omega)\),
horizontal isomorphism $\Xi$ is the one defined in {\em (i)}.
\end{thm}

We call the right-hand vertical composites in Theorem \ref{thm:main} (i) and (ii) the {\it dual string cobracket} and the {\it dual string bracket},
respectively.

Note that the condition on the \(S\)-action can be replaced with BV-exactness;
see \ref{defn:BV-exact} and \ref{rem:Strategy} for details.
It is also worth mentioning that the composite $B_{HH}\circ I$ is nothing but the cohomological Batalin--Vilkovisky (BV) operator $\Delta$ on the Hochschild homology of a DGA $\Omega$ if $\Omega$ is the polynomial de Rham algebra of a manifold or the classifying space of a connected Lie group. By abuse of terminology, we may call $B_{HH}\circ I$ the BV operator in general.


As mentioned above, under the isomorphisms $\Theta_1$ and $\Theta_2$ due to Jones in \cite[Theorem A]{J},
the loop cohomology \(H^*(LM)\) and the string cohomology \(H^*_{S^1}(LM)\) are identified with the Hochschild homology and
the negative cyclic homology of $\apl(M)$, respectively.
Thus, we have

\begin{thm}\label{thm:manifolds}
{\em (i)} The dual string cobracket for $BG$ described in Proposition \ref{prop:LBG} coincides with that in Theorem \ref{thm:main} (i) up to isomorphisms $\Theta_1$ and $\Theta_2$.  \\
{\em (ii)} 
Let $M$ be a simply-connected closed manifold of dimension $d$. Then
the dual $[ \ , \ ]^\vee$ to the string bracket in the sense of Chas and Sullivan on $M$
coincides with the dual string bracket in Theorem \ref{thm:main} (ii)
up to isomorphisms $\Theta_1$ and $\Theta_2$.


\end{thm}

In view of \cite[Theorem 4.1]{K-M}, Theorem \ref{thm:main} (i) and Theorem \ref{thm:manifolds} (i) allow us to compute the dual string cobracket
on $H^*_{S^1}(LBG; \K)$ explicitly if $\K$ is a field of characteristic zero; see Section \ref{sect:computations}.
We observe that the classifying space $BG$ is formal and then the $S$-action is trivial; see Corollary \ref{cor:BV} below.


Moreover, by dualizing Theorem \ref{thm:main} (ii) and Theorem \ref{thm:manifolds} (ii), we have Theorem \ref{thm:brackets} described below for computing the string bracket of a manifold.
It turns out that, in the rational case, the original string bracket can be formulated as the loop product followed by the BV operator
on the loop homology. Before describing our main result concerning a manifold, we need a notion of the  Batalin-Vilkovisky exactness.

\begin{defn} \label{defn:BV-exact} A DGA $\Omega$ is {\it Batalin-Vilkovisky exact (BV-exact)}
if
$\im \widetilde{B} = \ker \widetilde{B}$,
where the reduced operator
$\widetilde{B}\colon \widetilde{HH}_*(\Omega) \to \widetilde{HH}_*(\Omega)$,
is a restriction of Connes' $B$-operator
$B :=\pi \circ \beta : HH_*(\Omega) \to HH_*(\Omega)$.
We say that  a simply-connected space $M$ is {\it BV-exact} if the polynomial de Rham algebra $\apl(M)$ of $M$ is.
\end{defn}

\begin{rem}\label{rem:BV-B} Let $M$ be a simply-connected closed manifold. The result \cite[Proposition 2]{F-T08} implies that
the dual of the BV operator $\Delta' : H_*(LM) \to H_{*+1}(LM)$ \label{index:HomologicalBV} is identified with the operator $B$ in Definition \ref{defn:BV-exact} under the isomorphism $\Theta_1$ mentioned above.
Then, it follows  that a manifold $M$ is BV-exact if and only if
$\im \widetilde{\Delta'} = \ker \widetilde{\Delta'}$ for the reduced BV operator
$\widetilde{\Delta'} : \widetilde{H}_*(LM) \to  \widetilde{H}_{*+1}(LM)$.
\end{rem}


\begin{thm}\label{thm:BV-S} A simply-connected DGA $\Omega$ is BV-exact if and only if the reduced $S$-action on
$\widetilde{HC}^-_{*}(\Omega)$ is trivial.
\end{thm}

We refer the reader to Theorem \ref{thm:BV-S_general} for a generalization of the result.
An important example with trivial reduced \(S\)-action
is given by the following proposition due to Vigu\'e-Poirrier.

\begin{prop}[{\cite[Proposition 5]{VP}}]
  \label{prop:vigue-formal}
  If a simply-connected DGA \(\Omega\) is formal,
  then the reduced \(S\)-action on $\widetilde{HC}^-_{*}(\Omega)$ is trivial.
\end{prop}

By combining Theorem \ref{thm:BV-S} and Proposition \ref{prop:vigue-formal},
we have

\begin{cor}
  \label{cor:BV}
  If a simply-connected DGA \(\Omega\) is formal,
  then it is BV-exact. As a consequence, a simply-connected manifold whose rational cohomology is generated by a single element and the classifying space of a compact connected Lie group are BV-exact.
\end{cor}

We also have a generalization of the corollary; see Theorem \ref{thm:weightBVexact}.


\begin{rem}\label{rem:Strategy}
It follows from Theorem \ref{thm:BV-S} that
the condition on the $S$-action in Theorems \ref{thm:main} and \ref{thm:manifolds} may be replaced with the BV-exactness.
This implies that the string brackets are determined exactly with the loop (co)products and the BV operator on the Hochschild homology of a DGA $\Omega$ without dealing with the cyclic homology of $\Omega$ itself
provided
$\Omega$ is BV-exact. We observe that there is an isomorphism
\[
\widetilde{\Delta}: \widetilde{HH}_*(\Omega)/ \im \widetilde{\Delta}
=  \widetilde{HH}_*(\Omega)/ \ker \widetilde{\Delta}
\stackrel{\cong}{\longrightarrow} \im \widetilde{\Delta} = \ker \widetilde{\Delta}.
\]
\end{rem}

Dualizing Theorems \ref{thm:main} (ii) and  \ref{thm:manifolds} (ii), we have

\begin{thm}\label{thm:brackets}
Let $M$ be a simply-connected closed manifold and $\K$ a field of characteristic zero. Assume further that $M$ is BV-exact.
Then there exists a commutative diagram
\[
\xymatrix@C65pt@R30pt{
H_*^{S^1}(LM; \K)^{\otimes 2} \ar[d]_{[ \ , \ ]}^{\text{the string bracket}} \ar[r]_{\cong}^{\Phi\otimes \Phi} &
(\ker \widetilde{\Delta'} \oplus \K[u])^{\otimes 2} \ar[r]^{inc. \oplus 0}
 &  H_*(LM; \K)^{\otimes 2} \ar[d]_{\text{the loop product}}^{\bullet} \\
H_*^{S^1}(LM; \K)  \ar[r]_{\cong}^{\Phi}& (\ker \widetilde{\Delta'} \oplus \K[u]) &
H_*(LM; \K). \ar[l]^-{\Delta'}
}
\]
Here $\widetilde{\Delta'} : \widetilde{H}_*(LM; \K) \to \widetilde{H}_{*+1}(LM; \K)$ denotes the reduced $BV$ operator on the homology and $\Phi$ is the dual of the composite of the isomorphisms $\Theta_2$ and $\Xi$ described in Theorem \ref{thm:main}.
\end{thm}

The shifted homology ${\mathbb H}_*(LM) := H_{*+d}(LM)$ for an orientable closed manifold $M$ of dimension $d$ admits a BV algebra structure with the loop product $\bullet$ and the BV operator $\Delta'$; see \cite{C-S}. It turns out that the homology is endowed with a Gerstenhaber algebra structure whose Lie bracket (loop bracket) $\{ \ , \  \}$ is given by
\[
 \{a, b\} = (-1)^{|a|}(\Delta' (a\bullet b) - (\Delta' a)\bullet b - (-1)^{|a|}a\bullet (\Delta' b))
\]
for $a, b \in {\mathbb H}_*(LM)$. If $a$ and  $b$ are in the kernel of $\Delta'$, then  $\{a, b\} = (-1)^{|a|}\Delta' (a\bullet b)$.
Therefore, by virtue of Theorem \ref{thm:brackets}, we have

\begin{cor}\label{cor:stringbracket_loopbracket} Under the same assumption and notations as in Theorem \ref{thm:brackets},
the rational string bracket of the loop space $LM$ is regarded as a restriction of the loop bracket up to the isomorphism $\Phi$.
\end{cor}

\begin{rem} (i) Proposition \ref{prop:vigue-formal} implies that Theorems \ref{thm:main}, \ref{thm:manifolds} and  \ref{thm:brackets} are applicable to a formal simply-connected closed manifold. \\
(ii) It follows from \cite[Theorem 8.5]{Chen12} that the loop homology of an orientable closed manifold admits a gravity algebra structure extending the Lie algebra structure on the string homology. Theorem \ref{thm:brackets} may enable us to determine a gravity algebra structure on the string homology
of a BV-exact manifold $M$; see Example \ref{ex:GravityAlg_mfd}.
\end{rem}

\begin{rem} In general, the cyclic homology (additive K-theory \cite{F-T})
for a DGA does {\it not} appear as the singular homology of any topological space because the
homology is of ${\mathbb Z}$-grading.
We stress that, however, the cyclic homology is used to investigate the string brackets for a manifold and
the classifying space of a Lie group. In fact,
the horizontal isomorphism $\Xi$ in Theorem \ref{thm:main} factors through the cyclic homology of $\apl(M)_\sharp$.
\end{rem}

\begin{rem}
By using the description of the dual loop product Dlp in \cite[Theorem 2.3]{K-M-N} and
Theorem \ref{thm:main},
we may relate the dual of the string bracket to the cup product on $H^*(LM; \K)$ for a manifold $M$. In fact, the isomorphism $\Xi$ in Theorem \ref{thm:main} is a morphism of algebras if the $S$-action is trivial;
see \cite[Theorem 2.5]{K-Y}. We observe that the additive K-theory
$K^+(\Omega) :=HC_{*-1}(\Omega)$ for a chain algebra $\Omega$
is a graded algebra with the Loday-Quillen $*$-product in \cite{L-Q}; see
\cite[Proposition 1.1]{K-Y}.
\end{rem}

We relate the BV-exactness to a more familiar rational homotopy invariant.

\newcommand{\wtpart}[2]{#1_{(#2)}}
\newcommand{\wt}[1]{\operatorname{wt}(#1)}
\begin{defn}\label{defn:p.w.}
A simply-connected space $X$ admits  \textit{positive weights} if the Sullivan minimal model
 \((\wedge V, d)\) for $X$
has a direct sum decomposition \(V = \bigoplus_{i > 0} \wtpart{V}{i}\) satisfying
\(d(\wtpart{V}{i}) \subset \wtpart{(\wedge V)}{i}\). A nonzero element in \(\wtpart{V}{i}\) is said to have {\it weight} $i$ and
the weight on \(V\) is extended in a multiplicative way to \(\wedge V\).
For \(x \in \wtpart{(\wedge V)}{i}\), its weight is written by \(\wt{x} = i\).
\end{defn}

Many spaces admit positive weights.
\begin{enumerate}
  \item The Sullivan minimal model ${\mathcal M}(X)$ of a  formal space $X$ is given by
    the  bigraded model $(\Lambda V,d)$   of its cohomology algebra $H^*(X;\Q )$ \cite[\S 3]{H-S},
    whose lower degree is given by $dV_p\subset (\Lambda V)_{p-1}$ for $p>0$ and $dV_0=0$ .
    Then the space $X$ admits positive weights defined by  $\wt{v}:=|v|+p$ for  $v\in V_p$.
  \item If a space $X$ has a two stage Sullivan minimal model
    ${\mathcal M}(X)=(\Lambda (V_0\oplus V_1),d)$ with $dV_0=0$ and $dV_1\subset \Lambda V_0$,
    then $X$ admits positive weights defined by  $\wt{v}:=|v|+i$ for  $v\in V_i$.
    For example, a homogeneous space is such a space even if it is not formal;
    see also Section \ref{sect:nonformal-example} for such a manifold.
  \item It is known that smooth complex algebraic varieties admit positive weights coming from its  mixed Hodge structure \cite{Mo}. In the paper, the  Sullivan minimal models are discussed  over $\C$,
  but admitting  positive weights is reduced to that over $\Q$; see \cite[Theorem 2.7]{BMSS}.
\end{enumerate}

\begin{thm}
  \label{thm:weightBVexact}
  A simply-connected space $X$ admitting positive weights is BV-exact.
\end{thm}

A simply-connected space does not necessarily admit positive weights. In fact, there exist a four cell complex \cite[\S 4]{MT} and elliptic spaces \cite[\S 5]{AL} not admitting positive weights;
see also Appendix \ref{app:appA}. It is worth mentioning that every finite group is realized as the group of self-homotopy equivalences of a rationalized
elliptic space which does not admit
positive weights; see \cite{CV}.

\section{Preliminaries}\label{sect:HH-HC}
In this section, we recall the Hochschild homology and the cyclic homology together with relationships between them and the loop homology.

\subsection{Hochschild and cyclic homology}
\label{sect:HH-HC-1}
In this section we recall the definitions of the Hochschild chain complex and the cyclic bar complex in \cite{G-J} and \cite{G-J-P}.
Let $\Omega$ be a connected commutative DGA over a field $\K$ of arbitrary characteristic endowed with a differential $d$ of degree $-1$. We call a DGA $\Omega$ \textit{nonpositive} if $\Omega = \oplus_{i\leq 0}\Omega_i$. In what follows,
it is assumed that a DGA is nonpositively graded algebra with the properties above unless otherwise stated. The degree of a homogeneous element $x$ of a graded algebra is denoted by $\deg{x}$.

First we recall the Hochschild chain complex together with the Connes' \(B\)-operator.
Write $\overline{\Omega}=\Omega/{\K}$ and
$C(\Omega )=\sum_{k=0}^{\infty}\Omega\otimes \overline{\Omega}^{\otimes k}$.
We define \(\K\)-linear maps \(b, B\colon C(\Omega)\to C(\Omega)\) of degrees \(-1\) and \(1\) by
\begin{eqnarray*}
  b(w_0,\ldots,w_k)&=&-\sum_{i=0}^k(-1)^{{\epsilon}_{i-1}}(w_0,\ldots,w_{i-1},dw_i, w_{i+1},\ldots,w_k)\\[-5pt]
               &&\hspace{-3cm} -\sum_{i=0}^{k-1}(-1)^{{\epsilon}_i}(w_0,\ldots,w_{i-1},w_iw_{i+1},w_{i+2},\ldots,w_k) +  (-1)^{(\deg{w_i}-1){\epsilon}_{k-1}}(w_kw_0,\ldots,w_{k-1}),\\
B(w_0,\ldots,w_k)&=&\sum_{i=0}^k(-1)^{({\epsilon}_{i-1}+1)( {\epsilon}_{k}-{\epsilon}_{i-1})}(1,w_i,\ldots,w_{k},w_0,\ldots, w_{i-1}).
\end{eqnarray*}
Here $\operatorname{deg} (w_0,\ldots,w_k)=\deg{w_0}+\cdots +\deg{w_k}+k$ for $(w_0,\ldots,w_k)\in C(\Omega )$, $\epsilon_i=\deg{w_0}+\cdots +\deg{w_i}-i$ and $\deg{u}=-2$.
Note that the formulae $bB+Bb=0$ and $b^2=B^2=0$ hold.
The chain complex \((C(\Omega), b)\) is called the {\it Hochschild chain complex}.
The {\it Hochschild homology} \(HH_*(\Omega)\)\label{index:HochschildHomology}
and the {\it reduced Hochschild homology} \(\widetilde{HH}_*(\Omega)\)\label{index:reducedHochschildHomology}
are the homologies of the complexes \((C(\Omega), b)\) and \((C(\Omega)/\K, b)\), respectively.

The {\it cyclic bar complex}
is the complex $(C(\Omega )[u^{-1}],b+uB)$\label{index:cyclicBarComplex},
where $b$ and $B$ are regarded as \(\K[u^{-1}]\)-linear maps extending $b$ and $B$ on $C(\Omega)$.
Its homology is denoted by $HC_*(\Omega)$ and called the {\it cyclic homology}.
The {\it negative cyclic homology} $HC^-_*(\Omega )$,
the {\it reduced negative cyclic homology} $\widetilde{HC}^-_*(\Omega )$\label{index:reducedNegativeCyclicHomology} and
the {\it periodic cyclic homology} $HC^{{\rm per}}_*(\Omega )$\label{index:negativeCyclicHomology}
of a DGA $\Omega$ are defined as the homologies of the complexes
$(C(\Omega )[[u]],b+uB)$,
$((C(\Omega )/\K)[[u]],b+uB)$ and
$(C(\Omega )[[u,u^{-1}],b+uB)$, respectively.
Since a DGA in our case has negative degree, the power series algebra $C(\Omega )[[u]]$ coincides  with the polynomial algebra $C(\Omega )[u]$, similarly,
$(C(\Omega )/\K)[[u]]=(C(\Omega )/\K)[u]$ and
$C(\Omega )[[u,u^{-1}]=C(\Omega )[u,u^{-1}]$.

We recall Connes' exact sequences (\ref{eq:Connes'ExactSe}).
The projection of the cyclic complex onto itself gives rise to the map $S'$. More precisely, we have $S'(\sum_{i\geq0}x_iu^{-i}) =  \sum_{i\geq0}x_{i+1}u^{-i}$.
Observe that the cyclic homology $HC_*(\Omega)$ and the negative cyclic homology $HC_*^-(\Omega)$ are $\K[u]$-modules, where $\deg u = -2$. The multiplication $S=\times u : HC_{n+2}^-(\Omega) \to HC _n^-(\Omega)$ is called the $S$-{\it action}
\label{index:S-action}
on the negative cyclic homology.

For the connecting homomorphism $\beta$ in Connes' exact sequence (\ref{eq:Connes'ExactSe}),
we see that $\beta([a_0]) = [B(a_0)]$. Moreover, we have  $B_{HH}([(\sum_{i\geq0}x_iu^{-i})])= [B(x_0)]$
and $B_{HC}([(\sum_{i\geq0}x_iu^{-i})])=[B(x_0)]$.

\begin{rem}\label{rem:split}
Under the same notation as above,
the unit and augmentation of $\Omega$ yield a split exact sequence of $\K[u]$-modules of the form
$0 \to C(\K)[u] \to C(\Omega)[u] \to (C(\Omega)/\K)[u] \to 0$.
Then the splitting map $s' :  \widetilde{HC}^-_*(\Omega) \to HC^-_*(\Omega)$  gives rise to an  isomorphism $\spl : \widetilde{HC}^-_*(\Omega)\oplus \K[u] \stackrel{\cong}{\to} {HC}^-_*(\Omega)$.
We observe that  $C(\K)[u] = \K[u] = HC^-_*(\K)$.
\end{rem}

\subsection{Sullivan minimal models}
Let
${\mathcal M}(Z)=(\wedge {V},d)$
be the  Sullivan minimal model of a nilpotent CW complex $Z$ of finite type \cite{FHT}.
  It is a free $\Q$-commutative DGA 
 with a $\Q$-graded vector space $V=\bigoplus_{i\geq 1}V^i$
 where $\dim V^i<\infty$ and a decomposable differential in the sense that
 $d(V^i) \subset (\wedge^+{V} \cdot \wedge^+{V})^{i+1}$ and $d \circ d=0$.
 Here  $\wedge^+{V}$ denotes
 the ideal of $\wedge{V}$ generated by elements of positive degree.
Observe that  ${\mathcal M}(Z)$ determines the rational homotopy type of $Z$; that is,
the spatial realization $||{\mathcal M}(Z)||$ is homotopy equivalent to $Z_0$
the rationalization of $Z$.
In particular, we see that
\[V^n\cong {\rm Hom}(\pi_n(Z),\Q) \mbox{\ \ and\ \ }H^*(\wedge {V},d)\cong H^*(Z;\Q ).
\]
Here the second is an isomorphism of graded algebras.
Note that a space $X$ is said to be  {\it formal} if there exists  a quasi-isomorphism  $\rho : {\mathcal M}(X)\to (H^*(X;\Q ),0)$ of DGA's.
We refer the reader to \cite{FHT} for more details.

In what follows, let $\K$ be a field of characteristic zero unless otherwise specifically mentioned.
Let  ${\mathcal M}$ be a free DGA $(\wedge V,d)$ with $V=\oplus_{i>1}V^i$ over $\K$.
We denote by $(\HochschildModel, \delta, s)$ the double complex defined in \cite{B-V}.
Namely,   $\HochschildModel=\wedge (V\oplus \overline{V})$, $s$\label{index:s} is the unique derivation of degree $-1$ defined by $s(v)=\bar{v}$, $s(\overline{v})=0$
and $\delta$ is the unique derivation of degree $+1$ which satisfies $\delta\mid_V=d$ and $\delta s+s\delta =0$.
Here $\overline{V}$ is the suspension of $V$; that is $\overline{V}^{n}=V^{n+1}$.
By \cite[Theorem 2.4 (i)]{B-V}, we see that the map $\Theta :C({\mathcal M})\to \HochschildModel$ defined by
$\Theta (a_0,a_1,\ldots,a_p)=1/p!\ a_0s(a_1)\cdots s(a_p)$ is a chain map between the double complexes $(C({\mathcal M}),b,B)$ and
$(\HochschildModel, \delta, s)$.
Moreover, it follows from \cite[Theorem 2.4 (ii) and (iii)]{B-V} that the map $\Theta$ induces isomorphisms
 $H(\Theta) : HH_*({\mathcal M})=H_*(C({\mathcal M}),b)\cong H_*(\HochschildModel, \delta)$ and
 $H(\Theta \otimes 1) : HC^-_*({\mathcal M})=H_*(C({\mathcal M})[u],b+uB)\cong H_* (\HochschildModel[u], \delta+u\cdot s)$.

\begin{rem}\label{rem:beta-s}
As mentioned in Section \ref{sect:HH-HC-1}, the connecting homomorphism $\beta$ in Connes' exact sequence (\ref{eq:Connes'ExactSe}) is given by  $\beta([a_0]) = [B(a_0)]$. Therefore, it follows that $\beta([a_0]) = [s(a_0)]$ up to the isomorphism $H(\Theta)$; see again \cite[Theorem 2.4 (i)]{B-V}.
\end{rem}

Let $X$ be a simply-connected space of finite type and $LX$ the free loop space of $X$.
Then the Sullivan  minimal model of $LX$ over $\K$, ${\mathcal M}(LX) $,  is given by   $(\HochschildModel,\delta )$ \cite{V-S} \label{index:HochschildModel} and
the Sullivan  minimal model of the orbit space  $ES^1\times_{S^1}LX$, ${\mathcal M}(ES^1\times_{S^1}LX)$, is given by $(\cyclicModel, D):=(\HochschildModel[u], \delta+u\cdot s)$; \label{index:cyclicModel} see \cite[Theorem A]{V-B1}.
Thus we have isomorphisms $HH_*({\mathcal M}(X))\cong H^{-*}(LX; \K)$ and $HC^-_*({\mathcal M}(X))\cong H^{-*} (ES^1\times_{S^1}LX; \K)$ by composing $\Theta_1$ and $\Theta_2$ with $H(\Theta)$ and
$H(\Theta \otimes 1)$, respectively.


\subsection{Loop product and coproduct on Gorenstein spaces}
\label{subsect:loopprod-coprod}
In order to introduce uniformly the loop product due to Chas and Sullivan and the dual loop coproduct due to Chataur and Menichi, we recall the notion of a Gorenstein DGA introduced by F\'elix, Halperin and Thomas in \cite{FHT_G}.

Let $A$ be an augmented DGA over $\K$.  We call $A$ a {\it Gorenstein algebra} of dimension $d$
if
\begin{equation}
  \dim \ext_A^*(\K, A) =
  \begin{cases}
      0 & \text{if } * \neq d, \\
      1 & \text{if } * = d.
  \end{cases}
\end{equation}
Here $\ext$ is defined by using semifree resolutions;
see Appendix of \cite{FHT_G} for details.
A path-connected space $M$ is called a {\it Gorenstein space}
of dimension $d$ if the polynomial de Rham algebra $\apl(M)$ is a Gorenstein algebra of dimension $d$.

The result \cite[Theorem 3.1]{FHT_G} implies that a simply-connected Poincar\'e duality space, for example
a simply-connected closed orientable manifold of dimension $d$, is a Gorenstein space of dimension $d$.
It follows from \cite[Proposition 3.2]{FHT_G} that the classifying space $BG$ of a connected compact Lie group $G$ is also a
Gorenstein space of dimension $-\dim G$.
The following result due to F\'elix and Thomas is a key to defining the loop product and the loop coproduct on the loop homology of a Gorenstein space.

\begin{thm}{\em(}\cite[Theorem 12]{F-T}{\em)}\label{thm:ext}
Let $M$ be a simply-connected Gorenstein space of dimension $d$ whose cohomology with coefficients in $\Q$
is of finite type. Then
\[
\text{\em Ext}^k_{\apl(M^n)}(\apl(M), \apl(M^n)) \cong H^{k-(n-1)d}(M; \Q)
\]
for any integer $k$,
where $\apl(M)$ is considered an  $\apl(M^n)$-module via the diagonal map $\diag : M \to M^n$.
\end{thm}

For a Gorenstein space $M$ as in Theorem \ref{thm:ext}, let $D(\text{Mod-}\apl(M^n))$ be the derived category of right
 $\apl(M^n)$-modules. In the category,
we define $\diag^!$ by the map which
corresponds to a generator of the one dimensional vector space $H^0(M; \Q)$
under the isomorphism
$\ext^{(n-1)d}_{\apl(M^n)}(\apl(M), \apl(M^n)) \cong H^0(M)$.
Moreover, for a homotopy fibre square
$$
\xymatrix@C25pt@R15pt{
E' \ar[r]^{q} \ar[d]_{p'} & E \ar[d]^p \\
M^{} \ar[r]_\diag & M^{n} , }
$$
there exists a unique map $q^!$ in $\ext_{\apl(E)}^{(n-1)d}(\apl(E'), \apl(E))$
which fits into
the commutative diagram in $D(\text{Mod-}\apl(M^n))$
\[
\xymatrix@C25pt@R15pt{
\apl^*(E') \ar[r]^-{q^!} & \apl^{*+(n-1)d}(E)  \\
\apl^*(M^{}) \ar[r]_-{\diag^!} \ar[u]^{(p')^*}& \apl^{*+(n-1)d}(M^{n}) . \ar[u]_{p^*}
}
\]
The result follows from the same proof as that of \cite[Theorems 1 and 2]{F-T_ST}.

We recall the definition of the loop product on a simply-connected Gorenstein space $M$. Consider the diagram
\begin{equation}
  \label{eq:loopProd}
  \xymatrix@C25pt@R18pt{
    LM & LM\times_M LM \ar[l]_(0.6){\comp} \ar[d] \ar[r]^q & LM \times LM \ar[d]^{(ev_0,ev_0)} \\
    & M \ar[r]_{\diag} & M\times M, }
\end{equation}
where the right-hand square is the pull-back of the evaluation map $(ev_0, ev_0)$ defined by
$ev_0(\gamma) = \gamma(0)$ along the diagonal map $\diag$
and
$\comp$ denotes the concatenation of loops.
By definition, the composite
\[
q^!\circ (\comp)^* : \apl(LM) \to \apl(LM\times_M LM) \to \apl(LM \times LM)
\]
induces $\dlp$ the dual to the loop product on $H^*(LM; \Q)$; see \cite[Introduction]{F-T_ST}.

We define a product \(\bullet\) on ${\mathbb H}_{*}(LM) := H_{*+d}(LM)$, which is called the \textit{loop product}, \label{index:loopPro} by
$$
a\bullet b = (-1)^{d(\deg{a}+d)}((\dlp)^\vee)(a\otimes b)
$$
for $a$ and $b \in {\mathbb H}_*(LM)$;
see \cite[Proposition 4]{C-J-Y} and
\cite[Definition 3.2]{Tamanoi:cap products}.

In order to recall the loop coproduct for
a Gorenstein space $M$, we consider the commutative diagram
\begin{equation}
  \xymatrix@C25pt@R18pt{
    LM \times LM  & LM\times_M LM \ar[l]_(0.5){q} \ar[d] \ar[r]^(0.6){\comp} & LM \ar[d]^l \\
    & M \ar[r]_{\diag} & M\times M, }
\end{equation}
where $l : LM \to M\times M$ is a map defined by $l(\gamma)= (\gamma(0), \gamma(\frac{1}{2}))$.
By definition,  the composite
\begin{equation}
  \comp^!\circ q^* : \apl(LM\times LM ) \to \apl(LM\times_M LM)
  \to \apl(LM)
\end{equation}
induces the dual to the loop coproduct $\dlcop$ on $H^*(LM)$.
We define a product $\odot$ on the shifted cohomology $\mathbb{H}^*(LM)=H^{*-d}(LM)$, which is called the \textit{dual loop coproduct}, \label{index:Dcoproduct} by
\begin{equation}
  a \odot b=(-1)^{d(d-\deg{a})}\dlcop(a\otimes b)
\end{equation}
for $a\otimes b \in H^*(LM)\otimes H^*(LM)$.


\begin{rem}
The product $\bullet$ on ${\mathbb H}_{*}(LM)$
is associative and graded commutative if $M$ is a simply-connected Poincar\'e duality space (see \cite[Proposition 2.7]{K-M-N}).
So is the product $\odot$ on $\mathbb{H}^*(LM)$
if \(M\) is the classifying space \(BG\) of a connected Lie group \(G\)
(see \cite{C-M} and \cite[Theorem B.1]{K-M}).
Moreover, so are both of $\bullet$ and $\odot$ if \(M\) is a Gorenstein space with \(\dim(\bigoplus_n\pi_n(M)\otimes\Q) < \infty\) (see \cite[Theorem 1.1]{Naito} and \cite[Theorem 1.5]{W3}).
\end{rem}



\begin{rem} By the same fashion as above, a Gorenstein space is defined on an arbitrary field $\K$.
Then Theorem \ref{thm:ext} remains true after replacing $\apl(X)$ with the singular cochain algebra of $X$ with coefficients in $\K$. That is the original assertion in \cite{F-T_ST}. Moreover, the constructions of the loop product and the loop coproduct are applicable to the Gorenstein space $M$; that is, those products are defined
on the singular cohomology of $LM$ with coefficient in $\K$; see \cite{F-T_ST}.
However, we only use such an algebra defined on a field of characteristic zero for our purpose.
\end{rem}


We conclude this section with the definition of a BV algebra. In the next section, the nation plays an important role in defining the dual string cobracket of the classifying space of a Lie group.



\begin{defn}
A graded algebra $({\mathbb H}^*,\odot)$ equipped with an operator $\Delta$ on ${\mathbb H}^*$ of degree $-1$
is a {\it BV algebra} if $\Delta\circ\Delta=0$ and the {\it Batalin-Vilkovisky identity} holds;
that is,  for any elements $a$, $b$ and $c$ in
${\mathbb H}^*$,
\begin{align}
\Delta(a\odot b\odot c)&= \Delta(a\odot b)\odot c + (-1)^{\sdeg{a}}a\odot \Delta (b\odot c) +
(-1)^{\sdeg{b} \sdeg{a} + \sdeg{b}}b\odot\Delta(a\odot c) \\
&- \Delta(a)\odot b\odot c -(-1)^{\sdeg{a}}a\odot \Delta(b)\odot c
- (-1)^{\sdeg{a} + \sdeg{b}}a\odot b\odot\Delta(c),
\end{align}
where
$\sdeg{\alpha}$ stands for the degree of
an element $\alpha$ in ${\mathbb H}^*$.
\end{defn}

\section{Proofs of assertions}\label{sect:proofs}

The strategy of the proof of Proposition \ref{prop:LBG} is exactly that of \cite[Theorem 6.2]{C-S}. In order to make the sign computation more clear in our setting,
we give the proof.

\begin{proof}[Proof of Proposition \ref{prop:LBG}]
It is readily seen that the dual string cobracket satisfies skew-symmetry since the multiplication $m$ is commutative.
Indeed, we have
\[
[y,x]=(-1)^{\sdeg{y}}\beta (\pi (y) \odot \pi (x))
=(-1)^{\sdeg{x} (\sdeg{y} + 1) + 1}
\beta (\pi (x) \odot \pi (y))
=-(-1)^{\sdeg{x} \sdeg{y}}
[x,y].
\]

Let $\Delta : {\mathbb H}^*(LBG) \to {\mathbb H}^{*-1}(LBG)$ be the cohomological BV operator stated in \cite[Appendix E]{K-M}.
Remark that $\Delta$ coincides with the composite $\pi \beta$. It follows
from \cite[Corollary C.3]{K-M} that the triple $({\mathbb H}^*(LBG), \odot, \Delta)$ is a BV algebra,
and hence the bracket
$
\{ a, b \} := (-1)^{\sdeg{a} }\Delta(a\odot b) - (-1)^{\sdeg{a}}\Delta (a)\odot b - a\odot \Delta (b)
$
satisfies the Poisson identity:
\begin{equation}\label{eq:Poisson_identity}
\{ a, b\odot c \} = \{ a, b\} \odot c + (-1)^{(\sdeg{a} -1) \sdeg{b} } b\odot \{ a,c\}.
\end{equation}
In the case where $a=\pi(x)$, $b=\pi(y)$ and $c=\pi (z)$,
by applying $\beta$ to \eqref{eq:Poisson_identity},
we see that $\beta \{ \pi(x), \pi(y)\odot \pi (z) \}$ coincides with
\[
 \beta (\{ \pi(x), \pi(y)\} \odot \pi(z) + (-1)^{(\sdeg{\pi(x)} -1) \sdeg{\pi(y)} }  \pi(y)\odot \{ \pi (x), \pi(z) \}).
\]
Since $\Delta \pi =0$ and $\beta \Delta =0$, it follows that
  \begin{align}
      &
      \{ \pi(x), \pi(y)\}
      = (-1)^{\sdeg{x} -1}\Delta(\pi (x) \odot \pi (y))
      =-\pi [x,y]  \ \ \text{and} \\
      &
      \beta \{ \pi(x), \pi(y)\odot \pi (z) \}
      =-(-1)^{\sdeg{y} } \beta (\pi(x)\odot \pi [y,z])
      =-(-1)^{\sdeg{x} + \sdeg{y}}[x,[y,z]].
  \end{align}
Therefore, by combining the formulae, we see that
\begin{align}
&-(-1)^{\sdeg{x} + \sdeg{y}}[x,[y,z]]
\\
&= -\beta(\pi [x,y] \odot \pi(z)  ) - (-1)^{( \sdeg{\pi(x)} -1) \sdeg{\pi(y)}} \beta(\pi(y)\odot \pi [x,z])\\
&=-(-1)^{\sdeg{x} + \sdeg{y}}[[x,y], z]-(-1)^{( \sdeg{\pi(x)} -1) \sdeg{\pi(y)} + \sdeg{y}}[y,[x,z]]\\
&=(-1)^{ \sdeg{x} + \sdeg{y} +( \sdeg{x}+ \sdeg{y}) \sdeg{z}}
[z,[x,y]]
+(-1)^{\sdeg{x} ( \sdeg{y} +  \sdeg{z}) +  \sdeg{x} + \sdeg{y}}
[y,[z,x]].
\end{align}
Multiplying the both sides of the above equality by $(-1)^{ \sdeg{x} + \sdeg{y} +1+ \sdeg{x} \sdeg{z}}$, we have
\[
(-1)^{\sdeg{x} \sdeg{z}}[x,[y,z]]
=-(-1)^{ \sdeg{y}  \sdeg{z}}[z,[x,y]]-(-1)^{ \sdeg{x}  \sdeg{y}}[y,[z,x]]
\]
which is indeed the Jacobi identity. This completes the proof.
\end{proof}

\begin{proof}[Proof of Theorem \ref{thm:main}] We will first prove (i). We recall the homomorphisms $B_{HC} : HC_{n-1}(\Omega ) \to HC_n^-(\Omega)$ and $B_{HH} : HC_{n-1}(\Omega) \to HH_n(\Omega)$ in Connes' exact sequence in Section \ref{sect:HH-HC}, which are defined by $B_{HC}(\sum_{i\geq 0}x_iu^{-i})= Bx_0$ and $B_{HH}(\sum_{i\geq 0}x_iu^{-i})=Bx_0$.
The result \cite[Theorem 2.5 (i)]{K-Y} implies that $B$ is an isomorphism.
By assumption, the $S$-action is trivial. Then,
it follows from  \cite[Theorems 2.5 (ii)(iii)]{K-Y} that the map $I$ is an isomorphism.
By a direct calculation, we see that
$\pi\circ \Xi = B_{HH}\circ I$ and $\Xi\circ \text{`Cokernel'} = \beta$.
The same consideration as above enables us to obtain the result (ii).
\end{proof}

\begin{proof}[Proof of Theorem \ref{thm:manifolds}]
The assertions (i) and (ii) follow from \cite[Theorem A]{J}; see also \cite[Theorem 8.3]{Chen12}.
In fact,
the dual of the homology Gysin exact sequence for the fibration
$S^1 \to ES^1\times LM \to ES^1 \times_{S^1}LM$ is identified with the Connes exact sequence under isomorphisms $\Theta_1$ and $\Theta_2$ mentioned in the sentence before Theorem \ref{thm:manifolds}; see \cite[Theorem B]{BFG} and Appendix \ref{app:appB} for a description of the Gysin sequence in terms of rational models.
With those isomorphisms, we compare the dual to string bracket for a manifold and the dual string cobracket for $BG$ with the dual string bracket and the dual string cobracket in Theorem \ref{thm:main}, respectively.

To this end, we recall that
a simply-connected closed manifold $M$ of dimension $d$ is a Gorenstein space of dimension $d$. Moreover, the classifying space $BG$ of a connected compact Lie group $G$ is a Gorenstein space of dimension $d=-\dim \, G$; see \cite{FHT_G}. Thus the result
\cite[Theorem A]{F-T_ST} and observations in \cite[pages 419-420]{F-T_ST}
yield that the dual loop product $\bullet^\vee$ for the manifold $M$ and the dual loop coproduct $\odot$ for the classifying space $BG$
are nothing but the dual to the loop product and the dual to the loop coproduct, respectively. It turns out that the bracket on $H_*(LM; \K)$ for the manifold $M$
and the dual string cobracket on $H^*(LBG; \K)$ coincide with the original string brackets \eqref{eq:StringBracket} and
\eqref{eq:StringBracket2}, respectively.
Thus, we have the results.
\end{proof}


\begin{proof}[Proof of Theorem \ref{thm:brackets}]
Let $\Omega$ be the DGA $\Omega = A_{PL}(M)_\sharp \otimes_\Q\K$ for $M$.
We observe that the dual of the BV operator $\Delta' : H_*(LM; \K) \to H_{*}(LM; \K)$ on the homology is regarded as the BV operator
$\Delta: HH_*(\Omega) \to HH_*(\Omega)$
in Theorem \ref{thm:main}; see Remark \ref{rem:BV-B}.

Let $\widetilde{HH}_*$ denote the reduced Hochschild homology $\widetilde{HH}_*(\Omega)$. Dualizing the reduced BV operator
$\widetilde{\Delta'} : \widetilde{H}_*(LM) \to \widetilde{H}_*(LM)$, we have an exact sequence (*) :
$\widetilde{HH}_* \stackrel{\widetilde{\Delta'}^\vee}{\longrightarrow} \widetilde{HH}_* \stackrel{\pi}{\longrightarrow} \widetilde{HH}_*/\im \widetilde{\Delta'}^\vee \to 0$. Observe that
$\widetilde{\Delta'}^\vee = B_{HH}\circ I= \Delta$.
By considering the dual exact sequence of (*), we see that $\pi$ gives rise to the isomorphism
$\pi^\vee : \ker \widetilde{\Delta'} = \ker (\Delta^\vee) \stackrel{\cong}{\longrightarrow} (\widetilde{HH}_*/\im \Delta)^\vee$.
Theorem \ref{thm:manifolds} (ii) yields the result.
\end{proof}

%

In the rest of the section, we prove Theorems \ref{thm:BV-S} and \ref{thm:weightBVexact}.
First we prove the ``if'' part of Theorem \ref{thm:BV-S}.

\begin{proof}[Proof of the ``if'' part of Theorem \ref{thm:BV-S}]
  Let \(\Omega\) be a simply-connected DGA such that the reduced \(S\)-action on \(\widetilde{HC}^-_{*}(\Omega)\) is trivial.
  Consider the reduced version of Connes' exact sequence
  \begin{equation}
    \xymatrix@C18pt@R15pt{
      \cdots \ar[r] & \widetilde{HC}_{n+2}^-(\Omega) \ar[r]^-{S=0} & \widetilde{HC}_{n}^-(\Omega) \ar[r]^-{\pi} & \widetilde{HH}_{n}(\Omega) \ar[r]^-{\beta} & \widetilde{HC}_{n+1}^-(\Omega) \ar[r] &\cdots
    },
  \end{equation}
  which splits into a short exact sequence
  \begin{equation}
    \xymatrix@C18pt@R15pt{
      0 \ar[r] & \widetilde{HC}_{n}^-(\Omega) \ar[r]^{\pi} & \widetilde{HH}_{n}(\Omega) \ar[r]^-{\beta} & \widetilde{HC}_{n+1}^-(\Omega) \ar[r] & 0
    }.
  \end{equation}
  By definition, there is a decomposition
  \(\widetilde{B} =\pi \circ \beta : \widetilde{HH}_*(\Omega) \to \widetilde{HH}_*(\Omega)\)
  and hence the above short exact sequence implies
  \(\ker\widetilde{B} = \ker\beta = \im\pi = \im\widetilde{B}\).
\end{proof}

In order to prove the ``only if'' part of Theorem \ref{thm:BV-S},
we recall the notion of the proper exactness of a sequence of complexes defined in  \cite{Smith}.

\begin{defn}
  Let \(M_1\to M_2\to M_3\) be a sequence of complexes and chain maps (of arbitrary degrees).
  \begin{itemize}
    \item[i)] The sequence is \(H\)-exact at \(M_2\) if the sequence of cohomology
      \(H(M_1)\to H(M_2)\to H(M_3)\) is exact.
     \item[ii)] The sequence is \(Z\)-exact at \(M_2\) if the sequence of modules of cycles
      \(Z(M_1)\to Z(M_2)\to Z(M_3)\) is exact. 
    \item[iii)] \cite{Smith} The sequence is \textit{proper exact} (at \(M_2\))
      if the sequence is exact (as a sequence of underlying graded modules),
      \(H\)-exact and \(Z\)-exact.
    \item[iv)] The sequence is \textit{weakly proper exact} at \(M_2\)
      if the sequence is exact and \(H\)-exact.
  \end{itemize}
\end{defn}

The following lemma is useful to prove the proper exactness from the weak proper exactness of a given sequence.

\begin{lem}
  \label{lem:removeWeakly}
  Let \(M_0\xrightarrow{f_0} M_1\xrightarrow{f_1} M_2 \xrightarrow{f_2} M_3 \xrightarrow{f_3} M_4\) be a sequence of complexes which is
  proper exact at \(M_2\) and
  weakly proper exact at \(M_1\) and \(M_3\).
  Then it is proper exact also at \(M_3\).
\end{lem}

\begin{proof}
  For simplicity, we assume that the degrees of the chain maps are zero.
  We show that $\ker Z(f_3) \subset \im Z(f_2)$.
  For any $x_3$ in $\ker Z(f_3)$, there exists an element $y_2 \in M_2$
  such that \(f_2(y_2) = x_3\) by the exactness at \(M_3\).
  By the proper exactness at $M_2$,
  we see that $dy_2 = f_1(y_1)$ for some $y_1 \in Z(M_1)$.
  Since $H(f_1)[y_1] = [dy_2]=0$,
  it follows from the $H$-exactness at $M_1$
  that $y_1 -f_0y_0 = dz$ for some $[y_0] \in H(M_0)$ and $z \in M_1$.
  It is readily seen that $f_2(y_2 - f_1z) = f_2y_2 = x_3$ and $d(y_2 - f_1z) = 0$. We have the result.
\end{proof}


It is proved that the weak proper exactness for a long sequence yields the proper exactness.

\begin{prop}
  \label{prop:removeWeakly}
  A weakly proper exact sequence \(0\to M_0\to M_1\to M_2\to\cdots\) starting from \(0\)
  is always proper exact.
\end{prop}
\begin{proof}
  Since the sequence
  \(0\to 0\to 0\to M_0\to M_1\)
  is weakly proper exact at \(M_0\) and proper exact at \(0\), it follows from
 Lemma \ref{lem:removeWeakly} that the sequence is proper exact at \(M_0\).
  Similarly, the sequence
  \(0\to 0\to M_0\to M_1\to M_2\)
  gives proper exactness at \(M_1\).
  By repeating this argument, we can prove the proper exactness at \(M_n\) for all \(n\).
\end{proof}

\begin{rem}
  By the same argument as in the proof above,
  we can also prove the dual of Proposition \ref{prop:removeWeakly} which asserts that
  a weakly proper exact sequence ending with \(0\) is always proper exact.
\end{rem}

Next we give a key lemma for proving Theorem \ref{thm:BV-S}.

\begin{lem}
  \label{lem:Bexact}
  Let \(M_0\to M_1\xrightarrow{f_1} M_2\xrightarrow{f_2} M_3\) be a proper exact sequence.
  Then one has \(\im d\cap \ker f_2 = d(\ker f_2)\).
\end{lem}

\begin{proof}
The $Z$-exactness at $M_2$ and the $H$-exactness at $M_2$ give the result. The details are left to the reader.
\end{proof}

Note that the consequence in Lemma \ref{lem:Bexact} is equivalent to
the exactness of the sequence of modules of coboundaries.

Now we begin the proof of the ``only if'' part of Theorem \ref{thm:BV-S}.
Let \((\wedge V, d)\) be a Sullivan model of the DGA \(\Omega\) with \(V = V^{\geq 2}\).
Define \((\redHochschildModel, \delta) = (\wedge^{+} (V\oplus \overline{V}), \delta)\)\label{index:redHochschildModel}
and \((\redCyclicModel, D) = (\wedge u\otimes \redHochschildModel, D)\); see Section \ref{sect:HH-HC}.
Then \((\redHochschildModel, \delta)\) and \((\redCyclicModel, D)\)
are chain models for the reduced Hochschild homology and the reduced negative cyclic homology of \(\Omega\), respectively.
Let \(\redSusp\colon\redHochschildModel\to\redHochschildModel\) be the derivation defined by
\(\redSusp(v) = \bar{v}\) and \(\redSusp(\bar{v}) = 0\) for \(v\in V\).
Now we have a direct sum decomposition
\((\redHochschildModel, \delta) = \bigoplus_n (\redHochschildComponent{n}, \delta)\) of complexes,
where \(\redHochschildComponent{n} = \redHochschildModel \cap (\wedge V\otimes \wedge^n\overline{V})\).\label{index:redHochschildComponent}
Then \(\redSusp\) decomposes into a sequence
\(0\to\redHochschildComponent{0}\to\redHochschildComponent{1}\to\redHochschildComponent{2}\to\cdots\) of complexes.

\begin{lem}
  \label{lem:suspExact}
  The sequence
  \(0\to\redHochschildComponent{0}\to\redHochschildComponent{1}\to\redHochschildComponent{2}\to\cdots\)
  is exact; that is,
  \(\ker\redSusp = \im\redSusp\) in \(\redHochschildModel\).
\end{lem}
\begin{proof}
  Take a basis \(\{v_\lambda\}_\lambda\) of \(V\).
  Then we have
  \((\HochschildModel, \susp) \cong \bigotimes_\lambda(\wedge(v_\lambda, \bar{v}_\lambda), \susp)\)
  and hence
  \(H(\HochschildModel, \susp) \cong \Q\),
  which is equivalent to \(H(\redHochschildModel, \redSusp) \cong 0\).
\end{proof}

\begin{rem} \label{rem:B-s}
The operator $\widetilde{B} : \widetilde{HH}_*(\Omega) \to \widetilde{HH}_*(\Omega)$ is nothing but the homomorphism $H(\redSusp)$ up to the isomorphism $H(\Theta)$. This follows from the definition of the map $B$ in Section \ref{sect:HH-HC} and Remark \ref{rem:beta-s}.
\end{rem}

Now we recall a result of Vigu\'e-Poirrier which gives a description of the cyclic homology
in terms of \(\redHochschildModel\).
Here we give a proof for the convenience of the reader.

\newcommand{\inclKers}{\Phi}
\begin{lem}[{\cite[Lemma 2]{VP88}}]
  \label{lem:kersQuasiIsomCyclic}
  The canonical inclusion \(\inclKers\colon(\ker\redSusp, d) \to (\redCyclicModel, D)\)
  is a quasi-isomorphism.
\end{lem}
\begin{proof}
  Define bounded double complexes \(\{K^{p,q}\}\) and \(\{\redCyclicModel^{p,q}\}\) by
  \(K^{p,0}=(\ker\redSusp)^p\) and \(K^{p,q}=0\) for \(q\neq 0\), and
  \(\redCyclicModel^{p,q}=\wedge^qu\otimes\redHochschildModel^{p-q}\).
  Then their total chain complexes are
  \((\ker\redSusp, \delta)\) and \((\redCyclicModel, D)\), respectively,
  and the inclusion \(\inclKers\) gives rise to a morphism of double complexes.
  Now consider the filtration with respect to \(p\).
  By Lemma \ref{lem:suspExact}, we have
  \(E^{p,0}_1K = E^{p,0}_1\redCyclicModel = (\ker\redSusp)^p\) and
  \(E^{p,q}_1K = E^{p,q}_1\redCyclicModel = 0\) for \(q \neq 0\).
  Hence \(E_1\inclKers\) is an isomorphism and so is \(H\inclKers\)
  by the convergence of the spectral sequences.
\end{proof}

\newcommand{\connHom}{c}
Now we describe the $S$-action \(S = u\times(-)\colon H(\redCyclicModel)\to H(\redCyclicModel)\)
in terms of \(\ker\redSusp\).
By Lemma \ref{lem:suspExact}, we have an exact sequence
\(0\to\ker\redSusp\to\redHochschildModel\xrightarrow{\redSusp}\ker\redSusp\to 0\)
and its connecting homomorphism
\(\connHom\colon H(\ker\redSusp)\to H(\ker\redSusp)\) is given by
\(c([\redSusp\alpha])= [\delta\alpha]\).
Note that any element in \(H(\ker\redSusp)\) can be written as
\([\redSusp\alpha]\) for some \(\alpha\in\redHochschildModel\) with \(\delta\redSusp\alpha=0\),
since \(\ker\redSusp=\im\redSusp\) by Lemma \ref{lem:suspExact}.
By a straightforward computation, we have

\begin{lem}\label{lem:c-S}
  The map \(\connHom\) coincides with \(S\) through \(H\inclKers\) up to sign, i.e.,
  \(S\circ H\inclKers = -H\inclKers\circ\connHom\).
\end{lem}

We are ready to prove the ``only if'' part of Theorem \ref{thm:BV-S}.

\begin{proof}[Proof of the ``only if'' part of Theorem \ref{thm:BV-S}]
By Lemmas \ref{lem:kersQuasiIsomCyclic} and \ref{lem:c-S}, in order to prove the assertion,
  it suffices to show that the connecting homomorphism \(\connHom \) is trivial. To this end, we show that $[\delta \alpha] = 0$ in
 $H(\ker\redSusp)$ for any $\alpha \in \redHochschildModel$ with $\delta\redSusp \alpha = 0$; see the argument before
 Lemma \ref{lem:c-S}. Remark \ref{rem:B-s} yields that the BV-exactness of the DGA $\Omega$ is equivalent to the condition that
 the sequence (*):
 \begin{math}
   0\to\redHochschildComponent{0}\to\redHochschildComponent{1}\to\redHochschildComponent{2}\to\cdots
 \end{math}
is weakly proper exact. Thus, by Proposition \ref{prop:removeWeakly}, we see that the sequence (*) is proper exact. Moreover,
Lemma \ref{lem:Bexact} implies that \(\ker\redSusp\cap\im \delta=\delta(\ker\redSusp)\). Therefore, it follows that
 \(\delta\alpha\in\ker\redSusp\cap\im \delta=\delta(\ker\redSusp)\) for any
\(\alpha\in\redHochschildModel\) with \(\delta\redSusp\alpha=0\). We have the result.
\end{proof}




We conclude this section proving Theorem \ref{thm:weightBVexact}.
The proof is given by slightly modifying the proof of \cite[Proposition 5]{VP}.


\begin{proof}[Proof of Theorem \ref{thm:weightBVexact}]
  Recall that \((\redHochschildModel, \delta) = (\wedge^{+}(V\oplus \overline{V}), \delta)\)
  is a model of the Hochschild complex.
  For a derivation \(\theta\colon \wedge V\to \wedge V\) of degree \(0\) with
  \(\theta d = d\theta\) and \(\theta(V)\subset \wedge^{+}V\),
  define derivations \(L_\theta, e_\theta\colon \redHochschildModel\to\redHochschildModel\)
  by
  \(L_\theta(v) = \theta v\), \(L_\theta(\bar{v}) = \redSusp\theta v\),
  \(e_\theta(v) = 0\) and \(e_\theta(\bar{v}) = \theta v\).
  Then, as derivations on \(\redHochschildModel\), we have
  \([L_\theta, \redSusp] = [L_\theta, \delta] = [e_\theta, \delta] = 0\)
  and \([e_\theta, \redSusp] = L_\theta\).
  Hence \(L_\theta\) induces \(H(L_\theta)\colon H(\ker\redSusp)\to H(\ker\redSusp)\)
  and it follows that \(H(L_\theta)\circ\connHom = 0\colon H(\ker\redSusp)\to H(\ker\redSusp)\)
  by a straightforward computation from the above equations.

  Now we let \(\theta\) be the derivation defined by \(\theta(x) = \wt{x}x\)
  for weight-homogeneous elements \(x\in\wedge V\).
  Then for any weight-homogeneous element \(\alpha \in H(\ker\redSusp)\),
  we have \(0 = H(L_\theta)\circ\connHom(\alpha) = \wt{\alpha}\connHom(\alpha)\),
  where the weight on \(\redHochschildModel\) is defined
  as an extension of that on \(\wedge V\) with \(\wt{\bar{v}} = \wt{v}\) for \(v \in V\).
  By the positivity of the weight, we have \(\connHom(\alpha) = 0\) and hence \(\connHom = 0\).
  Therefore, Lemmas \ref{lem:kersQuasiIsomCyclic} and \ref{lem:c-S},
  imply the triviality of the reduced \(S\)-action,
  which is equivalent to the BV-exactness by Theorem \ref{thm:BV-S}.
\end{proof}

\section{The string brackets for formal spaces}\label{sect:computations}

In this section,
we consider string brackets for formal spaces as an application of Theorem \ref{thm:main}.

\subsection{Dual string cobrackets for classifying spaces}
We begin by considering the string bracket for the classifying space of a connected Lie group of rank one.

\begin{ex}\label{ex:rankOne}
The result \cite[Theorem 4.1]{K-M} enables us to compute the dual loop coproduct on the loop cohomology $\mathbb{H}^*(LBG; \mathbb{Q})$ for every compact connected Lie group $G$.
Thus, in particular, by Theorem \ref{thm:main}, we determine explicitly the Lie algebra structure of
$\mathcal{H}^*(LBSU(2)):=H_{S^1}^{*+ 3+ 1}(LBSU(2); \mathbb{Q})$ endowed with the dual string cobracket.  In fact, we see that
\begin{eqnarray*}
\mathcal{H}^*:=\mathcal{H}^*(LBSU(2))&\cong& (\widetilde{HH}_*(\Omega)/\im \Delta )_{-*-3-2}\oplus (\mathbb{Q}[u])_{-*-3-1} \\
&\cong& \mathbb{Q}\{x, x^2, \ldots , x^n, \ldots \}\oplus  \mathbb{Q}\{1, u, u^2, \ldots , u^k, \ldots \}
\end{eqnarray*}
as vector spaces, where $\Omega$ denotes
the Sullivan minimal model for $SU(2)$. Observe that $\deg{x^n}= 4n-5$ and $\deg{1}= -4$ for $x^n$ and $1 \in \mathcal{H}^*(LBSU(2))$.
The formula in \cite[Theorem 4.1]{K-M} for the loop product $\odot$ yields that $\Delta(x^n)\odot 1 = n x^{n-1}$, $\Delta(x^n)\odot \Delta(x^m) = \pm nm \Delta(x)x^{n+m-2}$ and $\Delta(1) = 0$ in
$\mathbb{H}^*(LBSU(2))$. Therefore, we see that $[1,1]=0$, $[x^n, x^m]=0$ for $m, n\geq 1$, $[u^l, \alpha]= 0$ for every $\alpha \in \mathcal{H}^*$, $l \geq 1$ and
$[x^n, 1]=-nx^{n-1}$ for $n\geq 1$.
\end{ex}

Next we consider the dual string cobracket for the classifying space of \(G\) with arbitrary rank.

\begin{prop}\label{prop:IteratedBrackets}
For each $n$, the $n$-fold dual string cobracket $[\calH, [\calH, \ldots , [\calH, \calH] \cdots ]]$ is non-trivial on $\calH^*:=H^{*+\dim G +1}_{S^1}(LBG;\K)$.
\end{prop}

\begin{proof}
For the case where $\rank G =1$, Example \ref{ex:rankOne} above implies the result. We assume that $N:=\rank G \geq 2$.
Recall the result \cite[Theorem 4.3]{K-M} which asserts that the loop cohomology ${\mathbb H}^*(LBG):=H^{*+\dim G}(LBG)$ is isomorphic to the tensor product of algebra $H^*(BG)\otimes H_{-*}(G) = \K[y_1, \ldots , y_N]\otimes \wedge(x_1^\vee, \ldots , x_N^\vee)$ equipped with the BV operator $\Delta$ given by
$\Delta(x_i^\vee x_j^\vee)= \Delta(y_iy_j)=\Delta(x_i^\vee) =\Delta(y_i) = 0$ and
\[
\Delta(y_i x_j^\vee) =
\left\{
\begin{array}{l}
0 \ \ \text{if} \ i\neq j,  \\
1  \ \ \text{if} \  i= j.
\end{array}
\right.
\]
Thus, an induction argument with the BV identity enables us to deduce that
\[
\Delta(y_1^{k_1}\cdots y_N^{k_N}x_{i_1}^\vee\cdots x_{i_s}^\vee)=
\sum_{1\leq j \leq s} (-1)^{d_j}k_{i_j}y_1^{k_1} \cdots y_{i_j}^{k_{i_j}-1}\cdots y_N^{k_N}x_{i_1}^\vee \cdots \widehat{x_{i_j}^\vee}\cdots x_{i_s}^\vee, 
\]
where $\widehat{\cdot}$ denotes the omission and $d_j = \deg{ x_{i_1}^\vee} + \cdots + \deg{x_{i_{j-1}}^\vee}$.
Therefore, it follows that
\[
\Delta(y_2x_2^\vee x_1^\vee)\odot\Delta(y_1^lx_1^\vee x_2^\vee) =
x_1^\vee\odot ly_1^{l-1}x_2^\vee = ly_1^{l-1}x_1^\vee x_2^\vee.
\]
Moreover, we see that $\Delta(y_1^{l-1}x_1^\vee x_2^\vee) \neq 0$ for $l\geq 2$. Then
the element $y^{l-1}x_1^\vee x_2^\vee$ is not in $\im \Delta$. Observe that $\Delta^2 =0$. We consider
an $n$-fold bracket of the form $\alpha := [y_2x_2^\vee x_1^\vee, [y_2x_2^\vee x_1^\vee, \ldots , [y_2x_2^\vee x_1^\vee, y_1^lx_1^\vee x_2^\vee]\cdots ]]$ for $l> n$.
It turns out that
\[
\alpha = l(l-1)\cdots (l-(n-1))y_1^{l-n}x_1^\vee x_2^\vee \neq 0
\]
in the codomain $(\widetilde{HH}_*(\Omega)/\im \widetilde{\Delta}) \oplus \K[u]$ of the dual string cobracket.
Theorem \ref{thm:main} (i) allows us to obtain the result.
\end{proof}


\subsection{String brackets for manifolds}
As an application of Theorem \ref{thm:manifolds} (or Theorem \ref{thm:brackets}),
we give another proof of the first half of the result \cite[Theorem 3.4]{B} due to Basu
and \cite[Example 5.2]{F-T-V07} due to F\'elix, Thomas and Vigu\'e-Poirrier.

\begin{prop}\label{prop:nil}
For a simply-connected closed manifold $M$ such that $H^*(M; {\mathbb Q})$ is generated by a single element,
the string bracket is trivial.
\end{prop}

\begin{proof} The result \cite[Theorem 1]{F-T08} implies that the loop homology of $M$ is isomorphic to the Hochschild cohomology of $\apl(M)$ endowed with the BV algebra structure due to Menichi \cite{Luc09}. We observe that $M$ is formal. Therefore, Theorem \ref{thm:brackets}
and explicit computations in \cite[Theorem 16]{Luc09-2} and \cite[Main Theorem]{Yang} yield the result.
In fact, for elements $\alpha_1$ and $\alpha_2$ in $\im \widetilde{\Delta} = \ker \widetilde{\Delta}$,
we have $\Delta(\alpha_1\bullet\alpha_2) = 0$; see Theorem \ref{thm:BV-S} and Remark  \ref{rem:BV-B}.
In particular, we observe the case where $H^*(M)\cong H^*(S^n)$ with $n$ odd. Then the generator $a_{-n}$ of the loop homology $\mathbb{H}_*(LM) := H_{*+n}(LM)$ with odd degree is in $H_0(LM)$. Then the generator $a_{-n}$ is not in $\ker \widetilde{\Delta}$; see Theorem \ref{thm:brackets}.
\end{proof}



The result \cite[Theorem 39]{Luc2} due to Menichi gives an explicit form of the BV operator on the rational loop homology of a connected compact Lie group.
We can also apply the result in our computation.  In particular, the behavior of the string bracket as seen in Proposition \ref{prop:nil} changes drastically
in case of a Lie group with rank greater than one.

\begin{prop} \label{prop:non-nil} (cf. \cite[Example 5.2]{F-T-V07}) Let $G$ be a simply-connected Lie group with rank greater than one.
The Lie algebra $\calH_*=H^{S^1}_{*-\dim G+2 }(LG; {\mathbb Q})$ endowed with the string bracket is non-nilpotent. More precisely, for any $n$, the $n$-fold bracket
$[\calH, [\calH, \ldots , [\calH, \calH] \cdots ]]$ is non-trivial.
\end{prop}

\begin{proof} We first observe that a simply-connected Lie group is formal.
Indecomposable elements $x_1, \ldots, x_N$ in ${\mathbb H}_*(G)$ are in the reduced homology $\widetilde{H}_*(LG)$ because $N:=\rank G > 1$.
Thus, it follows from \cite[Theorem 39]{Luc2} and \cite[Theorem 1]{Hepworth} that $x_i$ and $(s^{-1}x_j)^k$ are in $\ker \widetilde{\Delta}$.
Moreover, there exists a non-trivial $n$-fold string bracket. For example, for $k > n$, we see that on $\calH_*$,
\[
[x_j, [x_j, \ldots , [x_j, (s^{-1}x_j)^k] \cdots ]] = \pm k (k-1) \cdots (k- (n-1))(s^{-1}x_j)^{k-n} \neq 0.
\]
This follows from the explicit formula of the BV operator in \cite[Theorem 39]{Luc2} and Theorem \ref{thm:brackets}. Observe that $x_j$ is in $\ker \widetilde{\Delta}$ if $\rank G > 1$.
We have the result.
\end{proof}

\subsection{Gravity algebras}
The \textit{gravity algebra} with higher Lie brackets was introduced by Getzler \cite{G}.
We consider a gravity algebra structure which appears on the string homology of
a manifold and the classifying space of a Lie group; see, for example, \cite[Definition 8.1]{Chen12} for the definition of the gravity algebra.

\begin{ex}\label{ex:GravityAlg_BG} The result \cite[Theorem 1.1]{CEL19} due to Chen, Eshmatov and Liu
shows that the negative cyclic homology of
a DGA $\Omega$ admits a gravity algebra structure if the Hochschild homology of $\Omega$ has a BV algebra structure.
The higher Lie bracket
$[ \ , \ldots , \ ] : (HC_*^{-}(\Omega))^{\otimes n} \to HC_*^{-}(\Omega)$ is defined by
\[
[ x_1, \ldots , x_n ] =(-1)^{(n-1)\deg{x_1}+(n-2)\deg{x_2}+\cdots  + \deg{x_{n-1}}} \beta(\pi(x_1)\odot \pi(x_2) \odot \cdots \odot \pi(x_n))
\]
for $n \geq 2$, where $\odot$ denotes the dual loop coproduct on the Hochschild homology.

Let $G$ be a connected Lie group.
We see that all higher Lie brackets are non-trivial for the  classifying space $BG$. For the case where $\rank G \geq 2$, it follows from Theorem \ref{thm:main} that
$[ y_i x_i^\vee, \ldots , y_i x_i^\vee,  y_2x_2^\vee x_1^\vee, y_1^lx_1^\vee x_2^\vee] = \pm
1\odot \cdots \odot 1 \odot \Delta(y_2x_2^\vee x_1^\vee)\odot\Delta(y_1^lx_1^\vee x_2^\vee) =
ly_1^{l-1}x_1^\vee x_2^\vee \neq 0$ with the same notation as in the proof of Proposition \ref{prop:IteratedBrackets}.
Suppose that $\rank G = 1$. Then, with the same notation as in Example \ref{ex:rankOne}, we see that
$[x^2, \ldots , x^2] = \pm \coker(\Delta(x^2)\odot \cdots \odot \Delta(x^2)\odot \Delta(x^2)\odot 1) =
\pm \coker(\Delta(x^2)\odot \cdots \odot \Delta(x^2)\odot 2x)= \pm 2^{n-1}x^{n-1}\neq 0$
for the higher Lie bracket of rank $n$.
\end{ex}

\begin{ex}\label{ex:GravityAlg_mfd}
In \cite{Chen12}, Chen has proved that the string homology of an orientable closed manifolds admits a gravity algebra structure extending the Lie algebra structure; see \cite[Theorem 8.5]{Chen12} for more details.
Let $G$ be a simply-connected Lie group.
We see that all higher Lie brackets in the string homology of $G$ are non-trivial if and only if $\rank G > 1$.  In fact, in case of $\rank G > 1$,
by applying Theorem \ref{thm:brackets} to the higher Lie bracket of $G$, we have
$[x_j, s^{-1}x_j, \ldots, s^{-1}x_j ] = \pm k(s^{-1}x_j)^{k-1}$
in $H_*^{S^1}(LG; {\mathbb Q})$ with the same notation as in Proposition \ref{prop:non-nil}.
If $\rank G = 1$, the only generator $x_1$ of odd degree is not in $\ker \widetilde{\Delta}$ and then
all higher Lie brackets are trivial; see the computation in the proof of Proposition \ref{prop:non-nil}.
\end{ex}


\section{Computation of the string bracket for a nonformal space}
\label{sect:nonformal-example}
In this section, we consider the string bracket of a nonformal and BV-exact manifold. We begin recalling a nonformal manifold in
\cite[6.4 Example]{F-T-V07}.

Let $UTS^6 \to S^6$ be the unit tangent bundle over $S^6$. Then, we have a simply-connected $11$-dimensional manifold $M$ which fits in the pullback diagram
\[
\xymatrix@C18pt@R18pt{
M \ar[d] \ar[r] & UTS^6 \ar[d]^{p}\\
S^3 \times S^3 \ar[r]^-f & S^6 ,
}
\]
where $f : S^3\times S^3 \to S^6$ is a smooth map homotopic to the map defined by collapsing the 3-skeleton into a point.
Since the Euler class of the unit tangent bundle mentioned above is non-trivial, it follows that
the minimal model of $M$ has the form ${\mathcal M}=(\wedge (x, y, z), d)$, where $d(x) = 0 = d(y)$, $d(z) = xy$, $\deg{x}= \deg{y} = 3$ and $\deg{z} = 5$. It is readily seen that
$M$ is nonformal since the Massey product $\langle x,x,y\rangle$ does not vanish, see \cite[page 277]{H-S}.  Moreover, we have 

\begin{prop}\label{prop:11-mfd}
The $11$-dimensional manifold $M$ is BV-exact.
\end{prop}




Proposition \ref{prop:11-mfd} is proved by computing the Hochschild homology explicitly. To this end,
we recall the minimal model ${\mathcal M}$ for $M$ mentioned above.
The Hochschild homology of \({\mathcal M}\) is the homology of the Sullivan algebra
\(\HochschildModel = (\wedge(x, y, z, \bar{x}, \bar{y}, \bar{z}), d)\),
where \(d(\bar{x}) = 0 = d(\bar{y})\), \(d(\bar{z}) = -\bar{x}y + x\bar{y}\); see Section \ref{sect:HH-HC}.
To compute \(H(\HochschildModel)\), we define its subcomplex \(\HochschildSub\) by \(\wedge(x, y, \bar{x}, \bar{y}, \bar{z})\).
By a simple calculation, we have the following lemma.
\begin{lem}
  \label{lem:11-mfd-Hochschild-sub-basis}
  The set
  \begin{math}
    \{
    \bar{x}^p\bar{y}^q,\ %
    x\bar{x}^p\bar{y}^q,\ %
    y\bar{y}^q,\ %
    xy\bar{z}^r
    \mid
    p, q, r \geq 0
    \}
  \end{math}
  forms a basis of \(H(\HochschildSub)\).
\end{lem}

Next we compute \(H(\HochschildModel)\)
by comparing with \(H(\HochschildSub)\) and \(\HochschildModel/\HochschildSub\).

\begin{prop}
  \label{prop:11-mfd-Hochschild-basis}
  The following set forms a basis of the Hochschild homology \(H(\HochschildModel)\):
  \begin{align}
    \bar{x}^p\bar{y}^q,\ %
    x\bar{x}^p\bar{y}^q,\ %
    y\bar{y}^q,\ %
    xy\bar{z}^{r+1},\ %
    xz\bar{x}^p\bar{y}^q,\ %
    yz\bar{y}^q, \\
    xyz\bar{z}^r,\ %
    z\bar{x}^{p+1}\bar{y}^q - x\bar{x}^p\bar{y}^q\bar{z},\ %
    z\bar{y}^{q+1} - y\bar{y}^q\bar{z},
  \end{align}
  where \(p\), \(q\) and \(r\) run over all non-negative integers.
\end{prop}
\begin{proof}
Since there is an isomorphism of complexes
\((\HochschildModel/\HochschildSub, d) \cong (\Q\{z\}, 0) \otimes (\HochschildSub, d)\),
Lemma \ref{lem:11-mfd-Hochschild-sub-basis} implies that the set
\begin{math}
  \{
  z\bar{x}^p\bar{y}^q,\ %
  xz\bar{x}^p\bar{y}^q,\ %
  yz\bar{y}^q,\ %
  xyz\bar{z}^r
  \mid
  p, q, r \geq 0
  \}
\end{math}
forms a basis of \(H(\HochschildModel/\HochschildSub)\).
Consider the long exact sequence associated with the short exact sequence
\(0 \to \HochschildSub \to \HochschildModel \to \HochschildModel/\HochschildSub \to 0\).
The connecting homomorphism \(H(\HochschildModel/\HochschildSub) \to H(\HochschildModel)\)
sends \(z\) to \(xy\) and the other basis elements to zero.
Hence each basis element of \(H(\HochschildSub)\) or \(H(\HochschildModel/\HochschildSub)\)
corresponds to a basis element of \(H(\HochschildModel)\),
except for \(z\) and \(xy\).
By lifting basis elements of \(H(\HochschildModel/\HochschildSub)\)
to cocycles in \(\HochschildModel\),
we get the above basis.
\end{proof}

\begin{proof}[Proof of Proposition \ref{prop:11-mfd}]
  \newcommand{\Hs}{H\tilde{s}}
  Let \(\redHochschildModel\) be  the reduced complex
  \(\wedge^{+}(x, y, z, \bar{x}, \bar{y}, \bar{z})\).
  Recall that the reduced operation \(\widetilde{B}\) in
  Definition \ref{defn:BV-exact} is modeled by the map
  \(\Hs\colon H(\redHochschildModel) \to H(\redHochschildModel)\) induced by the derivation
  \(\tilde{s}\colon \redHochschildModel \to \redHochschildModel;\ v \mapsto \bar{v}\) (\(v = x, y, z\)); see Remark \ref{rem:beta-s}.

  By using the basis given in Proposition \ref{prop:11-mfd-Hochschild-basis}, we see that
  \begin{align}
    &\Hs(x\bar{x}^p\bar{y}^q) = \bar{x}^{p+1}\bar{y}^q,\ %
      \Hs(y\bar{y}^q) = \bar{y}^{q+1},\ %
      \Hs(xz\bar{x}^p\bar{y}^q) = z\bar{x}^{p+1}\bar{y}^q - x\bar{x}^p\bar{y}^q\bar{z}, \\
    &\Hs(yz\bar{y}^q) = z\bar{y}^{q+1} - y\bar{y}^q\bar{z} \ \  \text{and} \ \  %
      \Hs(xyz\bar{z}^r) = \frac{r+2}{r+1}xy\bar{z}^{r+1}.
  \end{align}
  This proves \(\ker\Hs = \im\Hs\).
\end{proof}

\begin{rem}\label{rem:Program}
A program \cite{W2} in a personal computer for computing the homology of a DGA helps us in proving Proposition \ref{prop:11-mfd-Hochschild-basis}. In fact, the computer calculation shows the basis in the proposition while our proof is {\it handmade}.
\end{rem}

\begin{rem}In the minimal model ${\mathcal M}$,
we define weights of $x$, $y$ and $z$ by $1$, $1$ and $2$, respectively.
Then, it is readily seen that the model ${\mathcal M}$ for the manifold $M$ admits positive weights. Therefore, Theorem \ref{thm:weightBVexact} enables us to conclude that $M$ is BV-exact. However, the explicit generators of the Hochschild homology of ${\mathcal M}$ represented in
Proposition \ref{prop:11-mfd-Hochschild-basis} are used in the computation below of the string bracket of $M$. We adhere to the proof of
Proposition \ref{prop:11-mfd}.
\end{rem}

The negative cyclic homology of ${\mathcal M}$ is isomorphic to the homology of  $\cyclicModel =(\wedge (u) \otimes \wedge (x,y,z,\bar{x},\bar{y}, \bar{z}) , D)$;
see Section \ref{sect:HH-HC}.
Here, the differential $D$ is given by
$D(x)=u\bar{x}$,
$D(y)=u\bar{y}$,
$D(z)=xy+u\bar{z}$,
$D(\bar{x})=D(\bar{y})=0$, $D(\bar{z})=-\bar{x}y+x\bar{y}$.
Then the morphism $\beta $ in Theorem \ref{thm:main} is induced by the derivation $s:\HochschildModel \to \cyclicModel$.

It follows from the BV-exactness of the manifold $M$ that $HC^{-}_* ({\mathcal M})$ decomposes into a direct sum $\Q[u] \oplus \im \tilde{\beta}$,
where $\tilde{\beta}$ is a morphism induced by the map $\tilde{s} :\tilde{\HochschildModel} \to \tilde{\cyclicModel}$ on the reduced complexes.
Hence, by applying $\tilde{\beta}$ to the basis except for $1$ in Proposition \ref{prop:11-mfd-Hochschild-basis},
we see that $\im \tilde{\beta}$ is spanned by the following homology classes:
\[
\zeta_{p,q} := \dfrac{1}{p!q!} \bar{x}^p \bar{y}^q ,
\hspace{0.5em}
\eta_{p,q} := \left\{
\begin{array}{ll}
\dfrac{1}{p!q!} (z\bar{x}^p \bar{y}^q - x\bar{x}^{p-1}\bar{y}^q\bar{z}) & (p \neq 0),\\
\dfrac{1}{q!} (z\bar{y}^q - y\bar{y}^{q-1}\bar{z}) & (p = 0),
\end{array}
\right.
\hspace{0.5em}
\theta _r := \dfrac{r+1}{r} xy\bar{z}^r
\]
for $p$, $q \geq 0$, $r\geq 1$ with $(p,q)\neq (0,0)$.
We also put $\zeta_{0,0}=1$ for convenience.
Denote by $\dsb$ the dual string bracket $[ \ , \ ]^{\vee}$ over $\Q$ stated in Theorem \ref{thm:manifolds}.

\begin{thm}\label{thm:An_explicit_computation}
For the dual string bracket $\dsb$ over $\Q$ of the $11$-dimensional manifold $M$, one has
  \begin{align}
  \dsb (\zeta_{p,q})=
  & \sum_{i=0}^{p+1} \sum_{j=0}^{q+1}
  \left\{ i(q+1)-j(p+1) \right\} ( \zeta_{i,j}\otimes \eta_{p+1-i, q+1-j}  -  \eta_{p+1-i, q+1-j}\otimes \zeta_{i,j}),
  \\
  \dsb(\eta_{p,q})=
  &
  \theta_2 \otimes \zeta_{p,q} - \zeta_{p,q}\otimes \theta_2
  - \sum_{i=0}^{p+1} \sum_{j=0}^{q+1}
  \left\{ i(q+1)-j(p+1) \right\} \eta_{i,j}\otimes \eta_{p+1-i,q+1-j}
   ,
  \\
\dsb (\theta _r) =
  & 0.
  \end{align}
\end{thm}

\begin{proof}
We first compute the dual loop product $\dlp$ by the rational model described in \cite{F-T-V07}.
Let ${\mathcal M}=\wedge V$ be the minimal model for $M$, $\pathspaceModel =(\wedge V)^{\otimes 2}\otimes \wedge \overline{V}$ the Sullivan model for the free path space stated in \cite[\S 15]{FHT} and $\varepsilon_{\pathspaceModel} : \pathspaceModel \to \wedge V$ the $(\wedge V)^{\otimes 2}$-semifiree resolution of $\wedge V$ which is given by the multiplication of $\wedge V$ and the canonical augmentation of $\wedge \overline{V}$.

By virtue of  \cite[Lemma 1]{F-T-V07}, we see that
a DGA morphism $\pathspaceModel \to \pathspaceModel \otimes _{\wedge V}\pathspaceModel$ defined by
$v_1 \otimes v_2 \mapsto v_1 \otimes 1 \otimes v_2$,
$\bar{x} \mapsto 1\otimes \bar{x} + \bar{x}\otimes 1$,
$\bar{y} \mapsto 1\otimes \bar{y} + \bar{y}\otimes 1$,
$\bar{z} \mapsto 1\otimes \bar{z} + \bar{z}\otimes 1- \frac{1}{2}\bar{x}\otimes \bar{y} + \frac{1}{2}\bar{y}\otimes \bar{x}$
for $v_i \in V$ is a Sullivan representative for the composition of free paths.
This induces a Sullivan representative ${\mathcal M}_{\comp} : \HochschildModel \to \HochschildModel \otimes _{\wedge V}\HochschildModel$ for $\comp$ in \eqref{eq:loopProd} which has formulae
  \begin{align}
  &{\mathcal M}_{\comp}(v)=v, \hspace{1em}
  {\mathcal M}_{\comp}(\bar{x})=1\otimes \bar{x} + \bar{x}\otimes 1,\hspace{1em}
  {\mathcal M}_{\comp}(\bar{y})=1\otimes \bar{y} + \bar{y}\otimes 1 \ \ \text{and}
  \\
  & {\mathcal M}_{\comp}(\bar{z})= 1\otimes \bar{z} + \bar{z}\otimes 1- \frac{1}{2}\bar{x}\otimes \bar{y} + \frac{1}{2}\bar{y}\otimes \bar{x},
  \end{align}
where $v\in V$.
Recall the morphism $\varepsilon_{\pathspaceModel }\otimes 1 : \pathspaceModel \otimes _{(\wedge V)^{\otimes 2}}\HochschildModel ^{\otimes 2} \to \wedge V\otimes _{(\wedge V)^{\otimes 2}}\HochschildModel ^{\otimes 2}$ appeared in the model for $\dlp$.
A section $\sigma$ of the morphism $\varepsilon_{\mathcal P}\otimes 1$ is given by
  \begin{align}
  &
  \sigma (v) =v\otimes 1, \
  \sigma(\bar{x}\otimes 1)= 1\otimes (\bar{x}\otimes 1), \
  \sigma(1\otimes \bar{x})= 1\otimes (1\otimes \bar{x}), \\
  &
  \sigma(\bar{y}\otimes 1)= 1\otimes (\bar{y}\otimes 1), \
  \sigma(1\otimes \bar{y})= 1\otimes (1\otimes \bar{y}), \
  \sigma(\bar{z}\otimes 1)= 1\otimes (\bar{z}\otimes 1), \\
  &
  \sigma(1\otimes \bar{z})= 1\otimes (1\otimes \bar{z}) -\bar{x}\otimes (1\otimes \bar{y}) + \bar{y}\otimes (1\otimes \bar{x}).
  \end{align}
Define a $(\wedge V)^{\otimes 2}$-morphism $\diag^! : \pathspaceModel \to (\wedge V)^{\otimes 2}$ of degree $11$ by
  \[
  \diag^! (1)=(-x\otimes 1 + 1\otimes x)(-y\otimes 1 + 1\otimes y)(-z\otimes 1 + 1\otimes z), \ \
  \diag^! | _{\wedge ^{+}\overline{V}} \equiv 0
  \]
which gives a representative of a nonzero element in ${\rm Ext}^{11}_{(\wedge V)^{\otimes 2}}(\wedge V , \wedge V)$; see \cite[Section 5]{W1} for the detail about a construction of the shriek map $\diag^!$.
Then, the result \cite[Theorem A]{F-T-V07} yields that the composite
  \[
  \xymatrixcolsep{3.8em}
  \xymatrix{
  \HochschildModel
  \ar[r]^-{{\mathcal M}_{\comp}}
  &
  \HochschildModel \otimes _{\wedge V}\HochschildModel
  \cong
  \wedge V\otimes _{(\wedge V)^{\otimes 2}}\HochschildModel ^{\otimes 2}
  \ar[r]^-{\sigma}
  &
  \pathspaceModel \otimes _{(\wedge V)^{\otimes 2}}\HochschildModel ^{\otimes 2}
  \ar[r]^-{\diag^! \otimes 1}
  &
  \HochschildModel ^{\otimes 2}
  }
  \]
induces the dual loop product $\dlp$ on homology.
This rational model and a straightforward computation enable us to compute $\dlp$ explicitly. In fact, we have
  \begin{align}
  &\dlp (\bar{x}^p\bar{y}^q)
  = xyz \otimes \bar{x}^p\bar{y}^q + \bar{x}^p\bar{y}^q \otimes xyz
  \\
  &\hspace{5.5em}
  +\sum_{i=0}^p \sum_{j=0}^q \binom{p}{i} \binom{q}{j}
  (
  -x\bar{x}^i \bar{y}^j \otimes yz\bar{x}^{p-i}\bar{y}^{q-j}
  +y\bar{x}^i \bar{y}^j \otimes xz\bar{x}^{p-i}\bar{y}^{q-j}\\
  &\hspace{14.5em}
  -xz\bar{x}^i \bar{y}^j \otimes y\bar{x}^{p-i} \bar{y}^{q-j}
  +yz\bar{x}^i \bar{y}^j \otimes x\bar{x}^{p-i} \bar{y}^{q-j}
  ),\\
  &\dlp  (z\bar{x}^p \bar{y}^q -x\bar{x}^{p-1}\bar{y}^q \bar{z})
  \\
  &\hspace{1em}
  =
  xyz\otimes (z\bar{x}^p \bar{y}^q -x\bar{x}^{p-1}\bar{y}^q \bar{z})
  +(z\bar{x}^p \bar{y}^q -x\bar{x}^{p-1}\bar{y}^q \bar{z})\otimes xyz
  \\
  &\hspace{2em}
  +xyz\bar{z}\otimes x\bar{x}^{p-1}\bar{y}^q +x\bar{x}^{p-1}\bar{y}^q \otimes xyz\bar{z}
  +xy\bar{z}\otimes xz\bar{x}^{p-1}\bar{y}^q - xz\bar{x}^{p-1}\bar{y}^q \otimes xy\bar{z}
  \\
  &\hspace{2em}
  -\sum_{i=0}^p \sum_{j=0}^q \binom{p}{i} \binom{q}{j}
  \left(
   xz\bar{x}^i \bar{y}^j \otimes yz\bar{x}^{p-i}\bar{y}^{q-j}
  -yz\bar{x}^i \bar{y}^j \otimes xz\bar{x}^{p-i}\bar{y}^{q-j}
  \right), \ \
  \\
  &\dlp (z\bar{y}^q-y\bar{y}^{q-1}\bar{z})\\
  &\hspace{1em}
  =-xyz\otimes (z\bar{y}^q - y\bar{y}^{q-1}\bar{z})
   -(z\bar{y}^q - y\bar{y}^{q-1}\bar{z}) \otimes xyz\\
   &\hspace{2em}
   +xyz\bar{z}\otimes y\bar{y}^{q-1} + y\bar{y}^{q-1}\otimes xyz\bar{z}
   + xy\bar{z}\otimes yz\bar{y}^{q-1} - yz\bar{y}^{q-1}\otimes xy\bar{z}\\
  &\hspace{2em}
  -\sum_{j=0}^q  \binom{q}{j}
  \left(
   xz\bar{y}^j \otimes yz\bar{y}^{q-j}
  -yz\bar{y}^j \otimes xz\bar{y}^{q-j}
  \right) \ \ \text{and}
  \\
  &\dlp  (xy\bar{z}^{r})
  = \sum_{i=0}^{r} \binom{r}{i}(-xyz\bar{z}^i \otimes xy\bar{z}^{r-i} + xy\bar{z}^i \otimes xyz\bar{z}^{r-i}).
  \end{align}
It follows from Theorem \ref{thm:manifolds} (ii) that
  \begin{align}
  &
  \dsb (\zeta_{p,q})= \dfrac{1}{p!q!} (\beta \otimes \beta) \bullet^\vee (\bar{x}^p\bar{y}^q),
  \\
  &
  \dsb(\eta_{p,q}) = \dfrac{1}{p!q!} (\beta \otimes \beta) \bullet^\vee  (z\bar{x}^p \bar{y}^q -x\bar{x}^{p-1}\bar{y}^q \bar{z}),
  \\
  &
  \dsb(\eta_{0,q}) = \dfrac{1}{q!} (\beta \otimes \beta) \bullet^\vee  (z \bar{y}^q -y\bar{y}^{q-1} \bar{z}) \ \ \text{and}\\
  &
  \dsb (\theta_r) = \dfrac{r+1}{r}(\beta \otimes \beta) \bullet^\vee  (xy\bar{z}^r).
  \end{align}
Therefore, by these formulae and the computations of $\dlp$ above, we have the result.
\end{proof}

\section{The cobar-type EMSS and $r$-BV-exactness}\label{sect:EMSS_rBVexact}

Let $X$ be a simply-connected space.
We define a cobracket on the cobar-type Eilenberg-Moore spectral sequence converging to the rational equivariant cohomology of the free loop space $LX$ which is compatible with
the dual to the string bracket in the sense of Chas and Sullivan \cite{C-S} if $X$ is a simply-connected closed manifold.

We begin by recalling the spectral sequence associated with a filtered complex $(A, F, d)$. Consider the submodules $Z_r^{p,q}$ and
$B_r^{p,q}$ defined by
\begin{equation}\label{eq:Z-B}
Z_r^{p, q} := F^pA^{p+q}\cap d^{-1}(F^{p-r}A^{p+q+1})  \  \ \text{and} \ \
B_r^{p, q}:= F^pA^{p+q}\cap d(F^{p-r}A^{p+q-1}).
\end{equation}
With the submodules of $A$, we have a spectral sequence $\{E_r, d_r\}$ whose $E_r$-term is defined  by $E_r^{p, q} := Z_r^{p, q} / (Z_{r-1}^{p+1, q-1}+ B^{p, q}_{r-1})$; see \cite[The proof of Theorem 2.6]{MCCleary}.

We use the same notation as that in Section \ref{sect:assertions}.  In particular, for a cochain algebra $A$, we define a chain algebra $A_\sharp$
by $(A_\sharp)_{-i} = A^i$ for $i$. The converse is also considered; that is, for a chain algebra $\Omega$, we have a cochain algebra
$\Omega^\sharp$ defined by $((\Omega)^\sharp)^i = \Omega_{-i}$ for $i$; see Remark \ref{rem:PositiveToNegative}.

Let \((\wedge V, d)\) be a Sullivan model of a simply-connected commutative cochain algebra \(A\). 
Define \((\HochschildModel, \delta) = (\wedge (V\oplus \overline{V}), \delta)\)
and \((\cyclicModel, D) = (\wedge u\otimes \HochschildModel, D)\); see Section \ref{sect:HH-HC}.
Then complexes \((\HochschildModel, \delta)\) and \((\cyclicModel, D)\) compute
the Hochschild homology and the negative cyclic homology of \(A_\sharp\), respectively.
Thus we have the cobar-type Eilenberg-Moore spectral sequence (the EMSS for short) $\{E_r^{*,*}, d_r\}$ converging to
$HC_*^-(A):=(HC_*^-(A_\sharp))^\sharp$ as an algebra
with
\[
E_2^{*,*} \cong \text{Cotor}_{\wedge (t)}^{*,*}(HH_*(A), \Q)
\]
as a bigraded algebra,
where $\deg t = 1$ and the $\wedge (t)$-comodule structure on the Hochschild homology $HH_*(A):=(HH_*(A_\sharp))^\sharp$ is induced by the derivation $s$ in the cyclic complex
$(\cyclicModel, D)$. In fact, the $\wedge (t)$-comodule structure $\nabla : \HochschildModel \to \HochschildModel \otimes \wedge (t)$
on  \((\HochschildModel, \delta)\) is given by $\nabla (\alpha) = \gamma(\alpha)\otimes t + \alpha \otimes 1$,
where $\gamma (\alpha) = (-1)^{\deg\alpha}s(\alpha)$. A map assigning the element
$au^n$ in the cyclic complex $\cyclicModel$ to an element  $a[t | \cdots  | t]$ in the $n$th cobar complex gives rise to an isomorphism
of complexes. As a consequence,  we have isomorphisms
\[
\text{Cotor}_{\wedge(t)}^*(\HochschildModel, \Q) \cong H(\cyclicModel, D) \cong HC_*^-(A_\sharp)^\sharp.
\]

\begin{rem}\label{rem:EMSS}
The isomorphisms above allow us to work in the category of $\wedge(t)$-comodule in order to investigate the negative cyclic homology of a DGA.
\end{rem}

We observe that, by construction, there is an isomorphism
$E_1^{0, *} \cong HH_*(A)$.
In particular, when we choose the polynomial de Rham algebra $\apl(M)$ for a simply-connected space $M$ as the DGA $A$, the spectral sequence converges to the $S^1$-equivariant cohomology $HC_{*}^-(A)\cong H^*_{S^1}(LM; \Q)$ with
\[
E_2^{*,*} \cong \text{Cotor}_{H^*(S^1; \Q)}^{*,*}(H^*(LM; \Q), \Q).
\]

\newcommand{\s}{N}
One has a direct sum decomposition
\((\redHochschildModel, \delta) = \bigoplus_n (\redHochschildComponent{n}, \delta)\) of complexes,
where \(\redHochschildComponent{n} = \redHochschildModel \cap (\wedge V\otimes \wedge^n\overline{V})\) with
$\HochschildModel = \redHochschildModel \oplus \Q$.
Then the reduced derivation \(\redSusp\) decomposes \((\redHochschildModel, \delta)\) into a sequence
\(0\to\redHochschildComponent{0}\to\redHochschildComponent{1}\to\redHochschildComponent{2}\to\cdots\) of complexes.
Thus it follows that the EMSS $\{E_r^{*,*}, d_r\}$ is decomposed as
\[
\{E_r^{*,*}, d_r\} = \bigoplus_{\s \in {\mathbb Z}}\{_{(\s)}E_r^{*,*}, d_r\} \oplus \{\Q[u], 0\},
\]
where $\text{bideg} \ u = (1,1)$, each spectral sequence $\{_{(\s)}E_r^{*,*}, d_r\}$ for $\s\geq 0$ is constructed by the double complex
\[
{}_{(\s)}{\mathcal K} : 0\to\redHochschildComponent{\s}\to\redHochschildComponent{\s+1}\otimes \Q\{u\} \to\redHochschildComponent{\s+2}\otimes \Q\{u^2\} \to\cdots
\]
and for $\s < 0$, the spectral sequence $\{_{(\s)}E_r^{*,*}, d_r\}$ is obtained by
the double complex
\[
{}_{(\s)}{\mathcal K} : 0\to 0 \to \cdots \to 0 \to \redHochschildComponent{0}\otimes \Q\{u^{-\s}\} \to\redHochschildComponent{1}\otimes \Q\{u^{-\s+1}\} \to\cdots.
\]
Here,
the double complex ${}_{(\s)}{\mathcal K}$ is regarded as a filtered complex associated with the horizontal degrees.
Thus, in the spectral sequence $\{_{(\s)}E_r^{*,*}, d_r\}$ for $\s < 0$, we have $_{(\s)}E_r^{i,*}=0$ for $i< -\s$.
We observe that each spectral sequence $\{_{(\s)}E_r^{*,*}, d_r\}$ converges the target as an algebra.

\begin{rem}\label{rem:HodgeDecom}
The direct sum of the targets of the spectral sequences $\{_{(\s)}E_r^{*,*}, d_r\}$ is nothing but the {\it Hodge decomposition} of
$HC_*^-(A)$; that is, we have
$\widetilde{HC}_*^-(A) = \oplus_{N\geq 0}H({\mathcal K}_{(N)})$; see \cite[Section 2]{BFG}.
If $A$ is the polynomial de Rham algebra $\apl(X)$ for a simply-connected space $X$, then the direct summands in the Hodge decomposition are identified with the eigenspaces of the Adams operation on
$\widetilde{H}^*_{S^1}(LX; \Q)$; 
see \cite[Theorem 3.2]{BFG} for the identification. We refer the reader to \cite[4.5.4]{Loday} for the operation.
 The result \cite[Theorem 1.1]{BRZ} shows that the string bracket respects
  the Hodge decomposition in some sense. Thus, we are also interested in computations of string brackets,
  as described in \ref{subsect:OpenQ} Problems,
  together with the consideration of the Hodge decomposition.
\end{rem}

\begin{prop}\label{prop:collapsing}
If the spectral sequence $\{_{(0)}E_r, d_r\}$ collapses at $E_r$-term, then so does $\{_{(\s)}E_r, d_r\}$ for each integer $\s$ and then
$\text{\em Tot} \, E_r\cong H^*_{S^1}(LX)$ as a vector space.
\end{prop}

Thus, it is readily seen that the collapsing of the EMSS is governed by that of the zeroth spectral sequence.

\begin{cor}\label{cor:collapsing}
The spectral sequence $\{_{(0)}E_r, d_r\}$ collapses at $E_r$-term if and only if so does $\{E_r^{*,*}, d_r\}$.
\end{cor}

\begin{lem}\label{lem:decom}
${_{(0)}E_r^{l+\s, *+\s}}\cong {_{(\s)}E_r^{l, *}}$ for $l \geq r-1$ and $\s\in \Z$.
\end{lem}

\begin{proof}
For $\s< 0$, the multiplication $u^{-\s}\times : {_{(0)}E_r^{*, *}} \to {_{(\s)}E_r^{*-\s, *-\s}}$ gives an isomorphism.
Assume that $\s\geq 0$.
By definition, we see that
\[
{_{(\s)}E_r^{l, *}}={_{(\s)}Z_r^{l, *}} / ({_{(\s)}Z_{r-1}^{l+1, *-1}}+ {_{(\s)}B^{l, *}_{r-1}}) \ \text{and}
\]
\[
{_{(0)}E_r^{l+\s, *+\s}}={_{(0)}Z_r^{l+\s, *+\s}} / ({_{(0)}Z_{r-1}^{l+\s+1, *+\s-1}}+ {{_{(0)}B^{l+\s, *+\s}_{r-1}}}),
\]
where
${}_{(\s)}Z$ and ${}_{(\s)}B$ denote the subcomplexes of ${}_{(\s)}{\mathcal K}$ defined in (\ref{eq:Z-B}) for the filtered complex
$A ={}_{(\s)}{\mathcal K}$.
Moreover, we have
${_{(\s)}Z_r^{l, *}} \cong {{_{(0)}Z_r^{l+\s, *+\s}}}$ and ${_{(\s)}Z_{r-1}^{l+1, *-1}} \cong {_{(0)}Z_{r-1}^{l+\s+1, *+\s-1}}$. Since $l \geq r-1$, it follows that
${_{(\s)}B^{l, *}_{r-1}} \cong {{_{(0)}B^{l+\s, *+\s}_{r-1}}}$.
Then the multiplication
$u^\s \times : {_{(\s)}E_r^{l, *}} \cong {_{(0)}E_r^{l+\s, *+\s}}$ is an isomorphism.
\end{proof}

\begin{proof}[Proof of Proposition \ref{prop:collapsing}]
Lemma \ref{lem:suspExact} yields that the spectral sequence $\{_{(0)}E_r^{*,*}, d_r\}$ converges to $0$ the trivial module.
The assumption and Lemma \ref{lem:decom} imply that ${_{(\s)}E_r^{l, *}}=0$ for $l \geq r-1$ and $\s$.
\end{proof}

\begin{thm}\label{thm:cobrackets} Let $M$ be a simply-connected closed manifold and
$A$ the polynomial de Rham algebra $\apl(M)$ of $M$.
Then the map
$[ \ , \  ]^\vee_r : E_r^{p, *} \to (E_r^{*, *}\otimes E_r^{*, *})^{p,*+d-2}$ defined by
$[ \ , \  ]^\vee_r\equiv 0$
for $p> 0$ and for $p=0$, the composite
\[
\xymatrix@C20pt@R20pt{
 E_r^{0,*}=\ker{d_{r-1}} \ar[r]^-{i} & HH_*(A) \ar[r]^-{\bullet^\vee} & HH_*(A)^{\otimes 2} \ar[r]^-{\Delta\otimes \Delta} &
E_r^{0,*}\otimes E_r^{0,*}
}
\]
gives rise to a cobracket on the spectral sequence, where $i$ denotes the inclusion. That is, it is compatible with the differentials and $H([ \ , \  ]^\vee_r) = [ \ , \  ]^\vee_{r+1}$.
Moreover, the cobracket $[ \ , \  ]^\vee_\infty$ is compatible with
the dual to the string bracket on $H^{S^1}_*(LM)$ at the $E_\infty$-term in the sense that the composite
\[
\xymatrix@C20pt@R20pt{
H_{S^1}^*(LM) \ar[r]^-{\pi} & E_\infty^{0,*} \ar[r]^-{i} & HH_*(A) \ar[r]^-{\bullet^\vee} & HH_*(A)^{\otimes 2} \ar[r]^-{\Delta\otimes \Delta} &
E_\infty^{0,*}\otimes E_\infty^{0,*}
}
\]
coincides with the dual to the string bracket
modulo $F^{1}H^*_{S^1}(LM)$.
Here $\pi$ is the projection and $\{F^{l}H^*_{S^1}(LM)\}_{l\geq 0}$ is the decreasing filtration associated with the spectral sequence.
\end{thm}

\begin{proof} By dimensional reasons, it is readily seen that
$(d_r\otimes 1\pm 1\otimes d_r)\circ[ \ , \  ]^\vee_r = 0 = [ \ , \  ]^\vee_r\circ d_r$
for $p>0$. Moreover, we see that every element in the image of $\Delta$ in $E_r^{0,*}$ is a permanent cocycle. In fact, for $w \in \im \Delta$, we have $Dw = (\delta + us)w = 0$. Then, it follows that
$(d_r\otimes 1\pm 1\otimes d_r)\circ[ \ , \  ]^\vee_r = 0 = [ \ , \  ]^\vee_r\circ d_r$ in
$E_r^{0, *}$. By the definition of the cobrackets, we have $H([ \ , \  ]^\vee_r) = [ \ , \  ]^\vee_{r+1}$.
In fact, the left-hand side is the restriction of $[ \ , \  ]^\vee_r$ in the non-trivial case.

Consider the compatibility of the cobracket at the $E_\infty$-term.
We have a commutative diagram
{\small
\[
\xymatrix@C15pt@R15pt{
HC^-_*(A) \ar[rr]^-{\pi} \ar[d]_{\cong}& &  HH_*(A) \ar[dd]^{=}\ar[r]^-{\bullet^\vee} & HH_*(A)^{\otimes 2}
\ar[r]^{\beta\otimes \beta} & HC_*^-(A)^{\otimes 2} \ar[d]^{pr}\\
HC_*^-(A)/ F^1\oplus  \ar[d]_{\cong}& \hspace{-1.15cm} F^1/F^2 \oplus \cdots \oplus F^* &  & & (HC_*^-(A)/F^1)^{\otimes 2} \ar[d]^{\cong}\\
\oplus_{p+q = *} E_\infty^{p,q} \ar[r]_-{pr} & E_\infty^{0,*}=\ker d_{*+1} \ar[r]_-i & HH_*(A) \ar[r]_-{\bullet^\vee} & HH_*(A)^{\otimes 2} \ar[r]_{\Delta\otimes \Delta} \ar[ruu]^{\beta\otimes \beta} &  E_\infty^{0,*}\otimes E_\infty^{0,*},
}
\]
}

\noindent
where $\{F^l\}_{l\geq 0}$ denotes the decreasing filtration of $HC_*^-(A)\cong H^*_{S^1}(LM)$ associated with the spectral sequence.
In fact, the commutativity of the left-hand side square and the right-hand side triangle follows from the construction of the spectral sequence; see, for example, \cite[The proof of Theorem 2.6]{MCCleary}. Theorem \ref{thm:manifolds} (ii) implies the upper sequence is the dual of the string bracket. We have the result.
\end{proof}


\begin{prop}\label{prop:S-action}
For each $N$, the $S$-action $u\times$ on $(\cyclicModel, D)$ gives rise to a map
\[
S : \{_{(\s)}E_r^{*,*}, d_r\} \to \{_{(\s-1)}E_r^{*+1,*+1}, d_r\}
\]
on the spectral sequence which is compatible
with the S-action on the negative cyclic homology $HC_*^-(A)$.
\end{prop}

\begin{proof}
The $S$-action on $(\cyclicModel, D)$ gives rise to a map
$_{(\s)}{\mathcal K} \to {}_{(\s-1)}{\mathcal K}$ which increases the filtration degree by $+1$ and is compatible with the differential. Then the map
induces the action $S$ on the spectral sequence.
\end{proof}

\begin{lem}\label{lem:EtoS}
Suppose that the $E_r$-term $_{(0)}E_r^{p,q}$ in $\{_{(0)}E_r^{*,*}, d_r\}$ is trivial for any $(p, q)$. Then the $(r-1)$ times $S$-action $S^{r-1} : \widetilde{HC}_*^-(A) \to \widetilde{HC}_*^-(A)$ is trivial.
\end{lem}

\begin{proof}
Let $x$ be an element in $\widetilde{HC}_*^-(A)$. Then $x$ is in $\widetilde{HC}_*^{-,(n)}(A)$ for some $n \geq 0$ and then it is represented by an element $\alpha$ in $_{(n)}E_\infty^{t,*}$ for some $t\geq 0$. Thus the element $S^{r-1}x$ is represented by
$S^{r-1}\alpha \in {}_{(n-(r-1))}E_\infty^{t+(r-1),*}$.
By assumption, it follows from Lemma \ref{lem:decom} that $_{(\s)}E_r^{l,*} = 0$ for $l \geq r-1$ and $\s \in\Z$.
This implies that $S^{r-1}\alpha = 0$ in the $E_\infty$-term and that there is no extension problem; that is,
$S^{r-1} x = S^{r-1} \alpha = 0$ in $HC_*^{-, (n-(r-1))}(A) \subset HC_*^-(A)$.
This completes the proof.
\end{proof}

Moreover, we have

\begin{thm}\label{thm:BV-S_general}
The $E_r$-term $_{(0)}E_r^{p,q}$ in $\{_{(0)}E_r^{*,*}, d_r\}$ is trivial for any $(p, q)$ if and only if
the $(r-1)$ times $S$-action $S^{r-1}$ on $\widetilde{HC}_*^-(A)$ is.
\end{thm}

\begin{proof}
  The ``only if'' part follows from Lemma \ref{lem:EtoS}.
  To prove the ``if'' part, we assume that \(S^{r-1}\) is trivial on \(\widetilde{HC}_*^-(A)\).
  Take any element
  \(x = x_p\otimes u^p + x_{p+1}\otimes u^{p+1} + \cdots \in {_{(0)}Z_r^{p,*}}\),
  where \(x_i \in \redHochschildComponent{i}\) is zero for sufficiently large \(i\).
  By the definition of \(_{(0)}Z_r^{p,*}\), the total differential increases the filtration degree of \(x\) by \(r\), i.e.,
  we have \(dx_p = 0\) and \(\redSusp x_i + dx_{i+1} = 0\) for \(p\leq i\leq p+r-2\).
  Now we have an element
  \([\redSusp x_{p+r-1}] \in H(\ker\redSusp)\)
  and the above equation implies
  \([dx_{p+1}] = S^{r-1}[sx_{p+r-1}] = 0 \in H(\ker\redSusp)\)
  by Lemma \ref{lem:c-S} and the assumption of triviality of \(S^{r-1}\).
  By Lemma \ref{lem:suspExact}, we see that
  \(\ker\redSusp = \im\redSusp\).
  Thus, there is an element \(v_p \in \redHochschildComponent{p}\) with \(d\redSusp v_p = dx_{p+1}\).
  By using these elements, we define
  \(y = (x_{p+1} - \redSusp v_p)\otimes u^{p+1} + x_{p+2}\otimes u^{p+2} + x_{p+3}\otimes u^{p+3} + \cdots \in {_{(0)}Z_{r-1}^{p+1,*}}\).
  Then we can show
  \(x - y = x_p\otimes u^p + \redSusp v_p\otimes u^{p+1} \in {_{(0)}B_{r-1}^{p,*}}\)
  by the same argument as above. It follows that \(x = y + (x - y) \in {_{(0)}Z_{r-1}^{p+1,*}} + {_{(0)}B_{r-1}^{p,*}}\)
  and hence \([x] = 0 \in {_{(0)}Z_r^{p,*}} / {_{(0)}Z_{r-1}^{p+1,*}} + {_{(0)}B_{r-1}^{p,*}} = {_{(0)}E_r^{p,*}}\).
  Since \(x\) is an arbitrary element of \(_{(0)}Z_r^{p,*}\),
  this proves the ``if'' part.
\end{proof}

The BV-exactness of a space is equivalent to the condition that the $E_2$-term of the spectral sequence $\{_{(0)}E_r^{*,*}, d_r\}$ is trivial. Then Theorem \ref{thm:BV-S_general} gives another proof of Theorem \ref{thm:BV-S}. This consideration allows us to propose a {\it higher version} of the BV-exactness.


\begin{defn}\label{defn:r-BV} A simply-connected space $X$ is {\it $r$-BV-exact} if
the $E_{r+1}$-term $_{(0)}E_{r+1}^{p,q}$ in the spectral sequence
$\{_{(0)}E_r^{*,*}, d_r\}$ associated with $X$ is trivial for any $(p, q)$.
\end{defn}

Indeed, there exists a non BV-exact space in the class of rational elliptic spaces; see Appendix \ref{app:appA} below.
While we are interested in the hierarchy of rational spaces defined by the $r$-BV-exactness as seen in \ref{subsect:OpenQ} Problems, we do not pursue the topic in this manuscript.

\medskip
\noindent
\section*{Acknowledgments}
The authors thank Jean-Claude Thomas and Luc Menichi for comments on the first draft of this manuscript. The authors are grateful to the referee for very careful reading of the manuscript, valuable suggestions and comments.

\appendix

\section{A non BV-exact space}\label{app:appA}

We describe an example of a non BV-exact space.
Let \((\wedge V, d)\)
be the minimal model \((\wedge (x_1, x_2, y_1, y_2, y_3, z), d)\)
of an elliptic space $X$ of dimension 228,
given in \cite[Example 5.2]{AL}.
The degrees are given by
\(\deg{x_1} = 10\), \(\deg{x_2} = 12\), \(\deg{y_1} = 41\), \(\deg{y_2} = 43\), \(\deg{y_3} = 45\) and \(\deg{z} = 119\).
The differential is as follows:
\begin{align}
  dx_1 &= 0 & dy_1 &= x_1^3 x_2
  & dz &= x_2(y_1 x_2 - x_1 y_2)(y_2 x_2 - x_1 y_3) + x_1^{12} + x_2^{10} \\
  dx_2 &= 0 & dy_2 &= x_1^2 x_2^2 \\
       & & dy_3 &= x_1 x_2^3
\end{align}
Note that $X$ does not admit positive weights.
Indeed, let $\wt{x_1}=i$ and $\wt{x_2}=j$.
Then $\wt{y_1}=3i+j$, $\wt{y_2}=2i+2j$ and $\wt{y_3}=i+3j$.
By $dz$, we have  the equations $5i+6j=12i=10j$
induced from  $\wt{x_2(y_1 x_2 - x_1 y_2)(y_2 x_2 - x_1 y_3)}=\wt{ x_1^{12}} =\wt{ x_2^{10}}$.
Thus we obtain  $i=j=0$.

Let \(\omega = x_1^{14}{y_2}{y_3} - x_1^{13}{x_2}{y_1}{y_3} + x_1^{12}x_2^{2}{y_1}{y_2}\)
be the representing cocycle of the fundamental class of the manifold,
which is considered as an element of \(\wedge^{+}V=\redHochschildComponent{0} \subset \redHochschildModel\).
Then we have
\([\omega] \notin \im(H(\redSusp)\colon 0 \to H(\redHochschildComponent{0})) = 0\).
On the other hand, we have
\([\omega] \in \ker(H(\redSusp)\colon H(\redHochschildComponent{0}) \to H(\redHochschildComponent{1}))\)
since \(\redSusp(\omega) = \delta (\alpha)\) for the element \(\alpha\) defined below.
Hence we have \(\im H(\redSusp) \subsetneq \ker H(\redSusp)\), i.e., \((\wedge V, d)\) is not BV-exact.
Note that we have found the element $\alpha$ by using the program \cite{W2} mentioned in Remark \ref{rem:Program} while the equality is also checked by hand.

Finally we consider the differentials in the spectral sequence \(\{_{(0)}E_r, d_r\}\)
defined in Section \ref{sect:EMSS_rBVexact}.
Since \(d_1[\omega] = H(\redSusp)[\omega] = 0\),
the cocycle \(\omega\) defines an element \([\omega] \in {_{(0)}E_2}\),
where \({_{(0)}E_2}\) is considered as a subquotient of \(_{(0)}E_1 = H(\redHochschildModel, \delta)\).
Then, the equality  \(\redSusp(\omega)=\delta(\alpha)\) enables us to compute
\(d_2[\omega] = [\redSusp\alpha] \neq 0 \in {_{(0)}E_2}\),
where the non-triviality is proved by using the program \cite{W2}.
Thus this Sullivan algebra gives an example such that \(d_2 \neq 0\) on \(_{(0)}E_2\).
Note that it is currently unknown whether  \(_{(0)}E_3 = 0\) (i.e., 2-BV-exact) or not.

\begin{align}
  \begin{autobreak}
    \alpha =
    -1380 x_1^{11}x_2^{6}\bar{y_3}
    - 5290 x_1^{11}x_2^{5}{y_3}\bar{x_2}
    - 114 x_1^{10}{y_1}{y_2}\bar{y_2}
    + 114 x_1^{10}{y_1}{y_3}\bar{y_1}
    - 114 x_1^{9}{x_2}{y_1}{y_2}\bar{y_1}
    + \frac{93}{2} x_1^{2}{y_2}{y_3}\bar{z}
    + x_1^{2}{y_2}{z}\bar{y_3}
    - x_1^{2}{y_3}{z}\bar{y_2}
    + 114 {x_1}x_2^{7}{y_2}{y_3}\bar{y_3}
    - \frac{93}{2} {x_1}{x_2}{y_1}{y_3}\bar{z}
    - {x_1}{x_2}{y_1}{z}\bar{y_3}
    - 114 {x_1}{x_2}{y_2}{z}\bar{y_2}
    + 115 {x_1}{x_2}{y_3}{z}\bar{y_1}
    + 113 {x_1}{y_1}{y_3}{z}\bar{x_2}
    + 572 {x_1}{y_2}{y_3}{z}\bar{x_1}
    + 115 x_2^{9}\bar{z}
    - 114 x_2^{8}{y_1}{y_3}\bar{y_3}
    + 114 x_2^{8}{y_2}{y_3}\bar{y_2}
    + 1150 x_2^{8}{z}\bar{x_2}
    + \frac{93}{2} x_2^{2}{y_1}{y_2}\bar{z}
    + 115 x_2^{2}{y_1}{z}\bar{y_2}
    - 115 x_2^{2}{y_2}{z}\bar{y_1}
    - 340 {x_2}{y_1}{y_2}{z}\bar{x_2}
    - 229 {x_2}{y_1}{y_3}{z}\bar{x_1}
  \end{autobreak}
\end{align}

\section{Connes' $B$-map in the Gysin exact sequence
}\label{app:appB}


In this section, by giving precisely a rational model for the {\it integration over the fibre} $\beta :H^{*+1}(LX; \Q) \to H^*_{S^1}(LX; \Q)$, we describe the Gysin exact sequence of the $S^1$-principal bundle $S^1 \to ES^1\times LX \stackrel{p}{\to} ES^1\times_{S^1} LX$ in terms of Sullivan models for $LX$ and $ES^1\times_{S^1} LX$. As a consequence, the Gysin sequence is identified with Connes' exact sequence
under the isomorphisms $HH_*({\mathcal M}(X))\cong H^{-*}(LX; \Q)$ and $HC^-_*({\mathcal M}(X))\cong H^{-*} (ES^1\times_{S^1}LX; \Q)$ described in Section \ref{sect:HH-HC}.

The cohomology Gysin sequence associated with the bundle has the form
\[
\xymatrix{
\cdots
\ar[r]
&
{H}_{S^1}^{*-2}(LX)
\ar[r]^-{S}
&
H_{S^1}^*(LX)
\ar[r]^-{p^*}
&
H^{*}(LX)
\ar[r]^-{\beta}
&
{H}_{S^1}^{*-1}(LX)
\ar[r]
&
\cdots
}
\]
in which $S$ is defined by the cup product with the Euler class $q^*(u)$, where $u$ is the generator of
$H^2(BS^1)$.
The $S^1$-principal bundle above fits in  the pullback diagram
\[
\xymatrix@C20pt@R15pt{
 & S^1 \ar@{=}[r] \ar[d]& S^1 \ar[d]\\
 LX \ar[r] \ar@{=}[d]& ES^1\times LX \ar[r] \ar[d]_{p}& ES^1 \ar[d]\\
  LX \ar[r]& ES^1\times_{S^1} LX \ar[r]_-{q} & BS^1
}
\]
in which the lower sequence is the fibre bundle associated with the  universal bundle
$ES^1 \to BS^1$.

We recall the Sullivan models $\HochschildModel$ and $\cyclicModel$ defined in Section \ref{sect:HH-HC}.
In the model $\cyclicModel$ for $ES^1 \times_{S^1} LX$, we write $u$ for the Euler class $q^*(u)$.
Thus the map $S$ in the Gysin sequence is regarded as the multiplication by $u$ in the models, namely the $S$-action in Connes' exact sequence.
In order to obtain rational models for $p$ and $\beta$, we here consider the relative Sullivan algebra $\HochschildModel^{\wedge}  := (\cyclicModel \otimes \wedge (e), \hat{\delta})$ with base $\cyclicModel$, where $\hat{\delta}(e)=u$, $|e|=1$.

\begin{lem}\label{lem:Gysin}
The canonical projection $\rho : \HochschildModel^{\wedge} \to \HochschildModel$ is a homotopy equivalence.
\end{lem}

\begin{proof}
We define a DGA morphism $\iota : \HochschildModel \to \HochschildModel^{\wedge} $ by $\iota (\alpha ) = \alpha + (-1)^{\deg{\alpha}} s(\alpha)e $ for $\alpha \in \HochschildModel$, where $s$ is the derivation on $\HochschildModel$ stated in Section \ref{sect:HH-HC}.
Then, we have $\rho \circ \iota = 1$ by definition.
Moreover, a homotopy $\HochschildModel^{\wedge} \to \HochschildModel^{\wedge} \otimes \wedge (t,dt)$ defined by $e \mapsto et$, $u \mapsto ut-edt$ and $\alpha \mapsto \alpha + (-1)^{\deg{\alpha}} s(\alpha)e(1-t)$ for $\alpha \in \cyclicModel$ implies that $\iota \circ \rho $ is homotopic to $1$.
This completes the proof.
\end{proof}

\begin{prop}
The derivation $s: \HochschildModel \to \cyclicModel$ is a rational model for $\beta$.
\end{prop}

\begin{proof}
From Lemma \ref{lem:Gysin}, the inclusion $\cyclicModel \hookrightarrow \HochschildModel^{\wedge }$ is a rational model for the principal bundle $p:ES^1 \times LM \to ES^1 \times_{S^1} LM$.
Since $\beta$ is the fiber integration associated with the principal bundle, it is modeled by a map $\int_e : \HochschildModel^{\wedge} \to \cyclicModel$ defined by $\int_e (\alpha_0 + \alpha_1 e) = \alpha_1$ for $\alpha_i \in \cyclicModel$.
Therefore, the result follows since the composite $\int_e  \iota$ coincides with the derivation $s$.
\end{proof}



\end{document}